\documentclass[a4paper,draft]{amsart}
 \usepackage{amsmath,amssymb,amsthm,latexsym}
  \theoremstyle{definition}
  \newtheorem{ddd}{Definition}[section]
  
  \theoremstyle{plain}
  \newtheorem{ttt}[ddd]{Theorem}
  \newtheorem{llll}[ddd]{Lemma}
  
  \theoremstyle{remark}

 \newcommand{\mdeg}{\mathrm{mdeg}}

 \newcommand{\supp}{\mathrm{supp}}

\begin{document}
  \title[Universal equivalence]{Universal Equivalence of Partially Commutative  Metabelian Lie Algebras}
  \author{E.\,N.\,Poroshenko, E.\,I.\,Timoshenko}
  \date{}
  \begin{abstract}
    In this paper, we find a criterium for universal equivalence of partially commutative Lie
    algebras whose defining graphs are trees. Besides, we obtain
    bases for partially commutative metabelian Lie algebras.
  \end{abstract}

  \maketitle

  \section{Introduction}

  Let
  $G=\langle X,E\rangle$ be an undirected graph without loops with the finite  set of vertices
  $X=\{x_1,\dots x_n\}$ and the set of edges
  $E$
  ($E\subseteq X\times X$). We denote the elements of
  $E$ by
  $\{x,y\}$.

  Consider a variety
  $\mathfrak M$ of Lie algebras over a ring
  $R$. A \emph{partially commutative Lie algebra in
  $\mathfrak M$ with a defining graph}
  $G$ is a Lie algebra
  $\mathcal{L}_R(X; G)$ defined as
  $$\mathcal{L}_R(X; G) = \langle X \,|\,  [x_i, x_j]=
      0 \Longleftrightarrow   \{x_i, x_j\} \in E; \,\,\,\,\mathfrak{M} \, \rangle$$
  in
  $\mathfrak{M}$. Thus, in this algebra, the variety identities and the defining relations
  hold together.  If there is no ambiguity denote this algebra just by
  $\mathcal{L}(X;G)$. For simplicity (to avoid using the notation for the set of edges), we write
  $\{x_i,x_j\}\in G$ instead of
  $\{x_i,x_j\}\in E$.

  Usually, the varieties whose identities do not
  imply additional relations of vertices' commutation are considered. It means that two
  vertices commute if and only if they are adjacent in
  $G$. In this paper, we study partially commutative metabelian Lie
  algebras. These algebras clearly possess the  property indicated above. It follows,
  for example, from the structure of the bases for partially commutative
  metabelian Lie algebras (see Theorem~%
\ref{baspcom}).

  Along with the variety of Lie algebras, one can consider other varieties
  of algebras and groups. The most actively studied objects of this kind
  are partially commutative groups which are defined by commutativity
  relations in the variety of all groups. Some papers (see,
\cite{GT09,Tim10,GT11} for example), are devoted to universal
  theories of partially commutative metabelian groups. In
\cite{CF69} and
 \cite{ML80}, partially commutative associative algebras were studied.

  This paper is organized as follows. In Sec.~%
\ref{prelim}, preliminary definitions and results are given.

  In Sec.~%
\ref{pcalgbas}, we find bases for partially commutative metabelian
  Lie (see Theorem~%
 \ref{baspcom}). This theorem is used a great deal in the paper but it
  is also interesting in itself.

  In Sec.~%
\ref{ann&cent}, we prove Theorem
 \ref{ann} and Theorem
 \ref{centrgen} which give information on the centralizers of some
  elements and on the annihilators of some elements in the derived
  subalgebra of partially commutative metabelian Lie algebras.
  We need these results for the study of the universal
  theories of partial commutative metabelian Lie algebras.

  The main result of the paper is Theorem~%
\ref{main} which is proved in Sec.~%
 \ref{univequiv}. This is a
  criterium for coincidence of the universal theories of partially
  commutative metabelian Lie algebras whose defining graphs are trees.
  It is easy to verify if this condition holds. So, the problem of universal
  equivalence for the algebras defined above  is algorithmically solvable.

  \section{Preliminaries}\label{prelim}

  Let
  $M(X)$ denote the free metabelian Lie
  $R$-algebra with the set of generators
  $X$. A \emph{partially commutative metabelian Lie algebra} generated by
  $X$ with the defining graph
  $G$ is the Lie algebra
  $M(X;G)=M(X)/I$, i.e. this is the Lie algebra, defined by the set of relations
  \begin{equation}\label{rel}
    [x_i,x_j]=0, \text{ if } \{x_i,x_j\} \in G
  \end{equation}
  in the variety of metabelian Lie algebras.
  \begin{ddd}
    Let
    $G$ be a graph. A vertex of
    $G$ is called an \emph{endpoint} if its degree is equal to
    $1$.
  \end{ddd}

  In the paper, we need a couple results with reference to free
  metabelian Lie algebras. Let us formulate these results here.
  \begin{ttt}\label{freeproduct}
\cite{Shme64}
    Let
    $L$ be the free polynilpotent Lie ring corresponding to the sequence
    $n_1,\dots, n_k$,
    ($k\geqslant 2$). If
    $x,y\in L$ are such that
    $[x,y]=0$ then either
    $x$ and
    $y$ are linearly dependent or
    $x,y\in L_{n_1,\dots, n_k}$.
  \end{ttt}
  We need this statement only in the case of a free metabelian Lie algebra.
  It is easy to see that Theorem~
\ref{freeproduct} also holds for a free metabelian Lie algebras
  over any integral domain.

  Usually, to study an algebra, it is very useful to know its basis. In
\cite{Bo63,Shme64}, linear bases for free polynilpotent Lie
  algebras were found. Since we need only the structure of a basis for metabelian Lie
  algebras, we give the description of a basis in this case only.

  \begin{ttt}\label{freemetabbas}
    A basis of the free metabelian Lie algebra
    $M(X)$ consists of the set
    $X$ together with Lie monomials of the form
    \begin{equation}\label{basis}
      [\dots[x_{i_1},x_{i_2}],\dots, x_{i_m}],\quad \text{ where } m\geqslant 2,\quad
      x_{i_2}<x_{i_1},\quad
      x_{i_2}\leqslant x_{i_3}\leqslant \dots \leqslant x_{i_m}.
    \end{equation}
  \end{ttt}
  Let us denote this basis by
  $\mathfrak{B}(X)$. Since it consists of left-normed Lie monomials only, in this paper, we omit
  all Lie brackets except the outer pair, i.e. we write
  $[x_{i_1},x_{i_2},\dots, x_{i_m}]$ instead of
  $[\dots[x_{i_1},x_{i_2}],\dots, x_{i_m}]$.

  By definition, we suppose
  $\ell([x_{i_1},x_{i_2},\dots,x_{i_r}])=r$ for any left-normed Lie
  monomial and say that
  $r$ is the \emph{length} of the monomial
  $[x_{i_1},x_{i_2},\dots,x_{i_r}]$.

  \begin{ddd} Let
    $u$ be a Lie monomial. The \emph{multidegree} of
    $u$ is the vector
    $\overline{\delta}=(\delta_1,\delta_2,\dots,\delta_n)$, where
    $\delta_i$ is the number of occurrences of
    $x_i$ in
    $u$.
  \end{ddd}
  Let us denote by
  $\mdeg (u)$ the multidegree of
  $u$. If all monomials of a Lie polynomial
  $g$ have the same multidegree
  $\overline{\delta}$ then we call such polynomial
  \emph{homogeneous} and write
  $\mdeg (g)=\overline{\delta}$.

  Note that we can define the sum of multidegrees as the sum of
  the corresponding vectors. We also use the notation
  $\mdeg_i (u)$ for the number of
  $x_i$ in
  $u$, i.e. for the
  $i$th coordinate of
  $\mdeg(u)$.

  Let us define an order on
  $\mathbb{Z}^n$, namely for arbitrary vectors
  $\overline{\delta}=(\delta_1,\dots,\delta_n)$ and
  $\overline{\gamma}=(\gamma_1,\dots,\gamma_n)$, let us set
  $\overline{\delta}> \overline{\gamma}$ if there is
  $k$ such that
  $\delta_k>\gamma_k$ for some
  $k$ but
  $\delta_i=\gamma_i$ for
  $i>k$. We use this order to compare multidegrees.
  Define an order on
  $\mathfrak{B}(X)$. Let
  $u,v \in \mathfrak{B}(X)$. We say that
  $u$ is \emph{greater} than
  $v$, if one of the following conditions holds:
  \begin{enumerate}
    \item
      $\mdeg(u)>\mdeg(v)$;
    \item
      $\mdeg (u)=\mdeg(v)$ and
      $u$ is greater  than
      $v$ lexicographically.
  \end{enumerate}
  We call this order by
  \emph{the standard order} or just \emph{the order} and denote it by
  ``$\geqslant$''. Let
  $f$ be an arbitrary Lie polynomial. Denote by
  $\overline{f}$ the largest monomial (in the sense of the standard order) appearing in
  $f$ without a coefficient by it. So,
  $f=\alpha \overline{f}+\sum_i \alpha_i u_i$,  where
  $\alpha,\alpha_i \neq 0$ and
  $\overline{f}> u_i$ for all
  $u_i$.

  Let
  $v=[x_{i_1},\dots,x_{i_k}]$. It is easy to see that any permutation of the letters
  $x_{i_3},\dots x_{i_k}$ gives us a monomial equal to
  $v$ in
  $M(X)$. Indeed,
  \begin{equation*}
    \begin{split}
      [x_{i_1},& x_{i_2}, \dots, x_{i_{s-1}},x_{i_{s}},x_{i_{s+1}}]\\
        & =[[x_{i_1},x_{i_2},\dots,x_{i_{s-1}}],[x_{i_{s}},x_{i_{s+1}}]]
          +[x_{i_1},x_{i_2},\dots,x_{i_{s-1}},x_{i_{s+1}},x_{i_{s}}]\\
        & =0+[x_{i_1},x_{i_2},\dots,x_{i_{s-1}},x_{i_{s+1}},x_{i_{s}}].
     \end{split}
  \end{equation*}
  Consequently,
  \begin{equation*}
    [x_{i_1}, x_{i_2}, \dots, x_{i_{s-1}},x_{i_{s}},x_{i_{s+1}},\dots, x_{i_k}]
       = [x_{i_1}, x_{i_2}, \dots, x_{i_{s-1}},x_{i_{s+1}},x_{i_{s}},\dots, x_{i_k}].
  \end{equation*}
  So, interchanging letters
  $x_{i_s}$ and
  $x_{i_{s+1}}$
  ($s\geqslant 3$) in
  $v$ gives us an element equal to
  $v$ in
  $M(X)$. We are left to note that the set of all transpositions permuting two neighbor elements
  generates the symmetric group on the set of these elements.
  Therefore,  we can obtain any permutation of the letters
  $x_{i_3},\dots, x_{i_k}$ in
  $v$ interchanging pairs of neighbor elements.
  \begin{equation}\label{permut}
     [x_{i_1},x_{i_2},x_{i_3},\dots x_{i_k}]=[x_{i_1},x_{i_2},x_{\sigma(i_3)},\dots, x_{\sigma(i_k)}],
  \end{equation}
  where
  $\sigma$ is a permutation on
  $\{i_3,i_4,\dots , i_k\}$.

  Let
  $R[X]$ be the set of all commutative associative polynomials over
  $R$. It follows from the last paragraph that the derived subalgebra
  $M'(X)$ of
  $M(X)$ is an
  $R[X]$-module with respect to the adjoint representation. Denote by
  $u.f$ the element of
  $M'(X)$ obtained by acting the element
  $f\in R[X]$ on
  $u\in M'(X)$. Namely, let us define
  $u.f$ inductively:
  \begin{enumerate}
    \item
      $u.y=[u,y]$ for any
      $y\in X$;
    \item
      Let
      $f=y_1y_2\dots y_m$ for
      $m\geqslant 2$ and let
      $f_0=y_1y_2\dots y_{m-1}$ then
      $u.f=(u.f_0).y_m$;
    \item
      Finally, if
      $f=g+s$, where
      $s$ is a commutative associative monomial then
      $u.f=u.g+u.s$.
  \end{enumerate}

  For any monomial in
  $R[X]$ we also can define its multidegree as follows:
  $\mdeg(\alpha x_1^{\gamma_1}x_2^{\gamma_2}\dots x_n^{\gamma_n})=(\gamma_1,\gamma_2,\dots,\gamma_n)$,
  if
  $\alpha \neq 0$.

  Let us remind the following well-known result. Let
  $g$ be a non-associative polynomial in
  $X$ and let
  $g=\sum_{\overline{\delta}}g_{\overline{\delta}}$, where
  $g_{\overline{\delta}}$ is a homogeneous polynomial such that
  $\mdeg(g_{\overline{\delta}})=\overline{\delta}$. If
  $g=0$ in
  $M(X;G)$ then
  $g_{\overline\delta}=0$ in this algebra for all
  $\overline{\delta}$. It follows from homogeneity of Lie algebras, metabelian
  algebra identities, and partial commutativity relations. Indeed, if
  $g=0$ then this polynomial can be nullified by using a finite set of identities
  and relations indicated above. Each step moves a homogeneous polynomial
  to homogeneous one of the same multidegree. Therefore, summands
  of different homogeneous components cannot cancel with each
  other.

  Then, we introduce some homomorphisms on
  $M(X;G)$. First, let us remind a well-known fact.
  Let
  $X=\{x_1,\dots,x_n\}$ be the set of generators of a Lie algebra
  $\mathcal{L}_1$. Consider a map
  $\varphi: X\rightarrow \mathcal{L}_2$, where
  $\mathcal{L}_2$ is also a Lie algebra. If all identities and relations of
  $\mathcal{L}_1$ hold under
  $\varphi$ then this map can be extended to a homomorphism  from
  $\mathcal{L}_1$ to
  $\mathcal{L}_2$ uniquely.
  So, all maps defined below are homomorphisms.

  Let
  $X$ be the set of generators of the partially commutative metabelian Lie algebra
  $M(X;G)$. For any non-empty subset
  $Y$ of
  $X$, denote by
  $G_Y$  the subgraph of
  $G$ generated by
  $Y$.
  \begin{ddd}
    The \emph{projection} from
    $M(X;G)$ onto
    $M(Y;G_Y)$ is a homomorphism
    $\pi_{Y}: M(X;G) \rightarrow M(Y;G_Y)$ defined on
    $X$ as follows:
    \begin{equation}
      \pi_{Y}(x_i)=
      \begin{cases}
        x_i, & \text{if } x_i\in Y\\
        0,   & \text{if } x_i\not \in Y.
      \end{cases}
    \end{equation}
  \end{ddd}
  \begin{ddd}
    Let
    $\overline{G}$ be obtained from
    $G$ by adding some edges. An \emph{identical simplification}
    from
    $M(X;G)$ onto
    $M(X;\overline{G})$ is the homomorphism
    $\varphi: M(X;G) \rightarrow M(X;\overline{G})$ defined on the
    set of generators identically (i.e.
    $\varphi(x_i)=x_i$ for
    $x_i\in X$).
  \end{ddd}

  Let
  $G$ be an arbitrary graph,
  $G_1,G_2,\dots ,G_r$ its connected components, and
  $X_i$ the set of the vertices of
  $G_i$ for
  $i=1,2,\dots r$. Given
  $t$ such that
  $0\leqslant t<r$, define the sets
  $X'=\bigcup_{i=1}^t X_i$,
  $X''=\bigcup_{j=t+1}^r X_j$. Consider a graph
  $\widetilde{G}$ obtained from the subgraph
  $G_{X'}$ adding some set
  $Z$ of isolated vertices
  ($Z$ may be empty). By
  $Y$ denote the set of vertices of
  $\widetilde{G}$. In other words,
  $Y=X'\cup Z$.

  \begin{ddd} \label{otozhd}
    An \emph{identification on connected components}
    $G_{t+1},\dots G_{r}$ is a homomorphism
    $\varphi:M(X;G)\rightarrow M(Y;\widetilde{G})$, acting on
    $X'$ identically and mapping all vertices in
    $X_k\subseteq X''$
    ($k=t+1,\dots , r$) to multiples of a fixed element
    $y_k\in Y$ (i.e. to elements of the form
    $\alpha y_k$, where
    $\alpha\in R$ can be different for different vertices in
    $X_k$). If
    $X'=\varnothing$ (this means
    $t=0$) then the identification of connected components is called  \emph{complete}.
  \end{ddd}
  Definition~%
\ref{otozhd} implies that if an identification of connected
  components is complete then
  $\widetilde{G}$ is totally disconnected, therefore
  $M(Y;\widetilde{G})=M(Y)$.

  Finally, let us recall some terminology related to universal theories of algebraic systems.
  \begin{ddd}
    An
    $\exists$-\emph{sentence} is a formula without free variables such that it is of the form
    $$\exists w_1\dots w_m \Phi(w_1,\dots,w_m),$$
    where
    $\Phi(w_1,\dots,w_m)$ is a formula of predicate calculus in
    the corresponding algebraic system such that this formula does not contain quantifiers.
  \end{ddd}

  \begin{ddd}
    The set of all
    $\exists$-sentences that are true in a Lie algebra
    $L$ is called the \emph{existential theory} or the
    $\exists$-\emph{theory} of this Lie algebra.
  \end{ddd}
  \begin{ddd}
    Lie algebras are called \emph{existentially equivalent}
    if their existential theories coincide.
  \end{ddd}
  The notion of \emph{universal theory}(or
  $\forall$-\emph{theory}) of a Lie algebra is defined analogously as well as
  the notion of universal equivalence of Lie algebras.

  It is easy to see that Lie algebras
  $L_1$ and
  $L_2$ are existentially equivalent if and only if these Lie algebras are
  universally equivalent.

  The procedure of exchanging functional symbols by predicate ones is well-known in model theory.
  Any set with all predicates induced on it is a submodel.

  Let us formulate a well-known result in model theory.
  \begin{ttt}\label{univeq}
    Arbitrary algebraic systems (ex., Lie algebras)
    $L_1$ and
    $L_2$ are universally equivalent if and only if each finite model of the first algebraic system
    is isomorphic to a finite model of the second one.
  \end{ttt}

  \section{Bases of partially commutative metabelian algebras}\label{pcalgbas}

   Given
   $u=[x_{i_1},x_{i_2},\dots,x_{i_k}]$, let
   $X_u$ be the set of all generators appearing in
   $u$. Let's denote by
   $G_u$ the subgraph of
   $G$, generated by the set
   $X_u$. Similarly, we define
   $X_g$ and
   $G_g$ for a homogeneous Lie polynomial
   $g$, i.e. suppose that
   $X_g=X_u$ and
   $G_g=G_u$, where
   $u$ is a monomial of
   $g$.

   \begin{llll}\label{equality}
     Suppose
     $u=[x_{i_1},x_{i_2},\dots, x_{i_k}]$,
     $x_{i_1}$ and
     $x_{i_j}$ are vertices belonging to the same connected component of
     $G_u$, and
     $u'$ is obtained from
     $u$ by interchanging
     $x_{i_1}$ and
     $x_{i_j}$. Then
     $u=u'$ in
     $M(X;G)$.
   \end{llll}
   \begin{proof}
     Since
     $x_{i_1}$ and
     $x_{i_j}$ are in the same connected component of
     $G_u$, there is a simple path
     $(x_{i_1},y_1,y_2,\dots,y_s,x_{i_j})$ in this graph, i.e. a path consisting of
     elements of
     $X_u$ all vertices of which are different. Let us proceed by induction on the number
     of intermediate vertices on such path (i.e. by induction on
     $s$).

     If
     $s=0$ then
     $x_{i_1}$ and
     $x_{i_j}$ are adjacent. Consider two cases.
     For
     $j=2$ we have
     $[x_{i_1},x_{i_2},\dots, x_{i_k}]=[0,x_{i_3},\dots,x_{i_k}]=0$
     and, analogously,
     $[x_{i_2},x_{i_1},\dots, x_{i_k}]=[0,x_{i_3},\dots,x_{i_k}]=0$.

     Let
     $j\geqslant 3$. By
(\ref{permut}), without loss of generality it can be assumed that
     $j=3$. Since
     $[x_{i_1},x_{i_3}]=0$ we obtain
     \begin{equation*}
       \begin{split}
         [x_{i_1},x_{i_2},x_{i_3},x_{i_4},\dots, x_{i_k}] =&
          [x_{i_1},x_{i_3},x_{i_2},x_{i_4},\dots, x_{i_k}]+[\dots [x_{i_1},[x_{i_2},x_{i_3}]],\dots, x_{i_k}] \\
           =&-[x_{i_2},x_{i_3},x_{i_1},x_{i_4},\dots,x_{i_k}]\\
           =&[x_{i_3},x_{i_2},x_{i_1},x_{i_4},\dots,x_{i_k}].
       \end{split}
     \end{equation*}

     Next, let us suppose that the statement holds for all
     $l<s$. Consider the path
     $(x_{i_1},y_1,\dots,y_s,x_{i_j})$ from
     $x_{i_1}$ to
     $x_{i_j}$ in
     $G_u$. As it was shown above, the monomial
     $u_1$ obtained by interchanging
     $x_{i_1}$ and
     $y_1$ in
     $u$ is equal to
     $u$ in
     $M(X;G)$. Therefore,
      $$[x_{i_1},x_{i_2},\dots, y_1,\dots, x_{i_k}]=[y_1,x_{i_2},\dots, x_{i_1},\dots, x_{i_k}]$$
     in
     $M(X;G)$. We have
     $(y_1,\dots,y_s,x_{i_j})$ contains
     $s-1$ intermediate vertices. So, by inductive assumption,
     $u_2=u_1$, in
     $M(X;G)$ where
     $u_2$ is the monomial obtained by interchanging
     $y_1$ and
     $x_{i_j}$ in
     $u_1$.

     Finally, to obtain the monomial
     $u'$ we are left to interchange
     $x_{i_1}$ and
     $y_1$ in
     $u_2$. If either
     $y_1$ or
     $x_{i_1}$ is the second letter in
     $u_2$ then one can interchange the first two letters by the anticommutativity
     identity, after that interchange
     $y_1$ and
     $x_{i_1}$ by the induction hypothesis, and again interchange the first two letters.
     If neither
     $y_1$ nor
     $x_{i_1}$ is the second letter of
     $u_2$ then
     $u_2=u'$ in
     $M(X;G)$ by
(\ref{permut}). Therefore,
     $u=u'$ in
     $M(X;G)$.
   \end{proof}
   \begin{llll}\label{zero}
     Let
     $u=[x_{i_1},\dots,x_{i_k}]$ be a monomial of the form
(\ref{basis}). Then
     $u$ is equal to zero in
     $M(X;G)$ if and only if the vertices
     $x_{i_1}$ and
     $x_{i_2}$ are in the same connected component of
     $G_u$.
   \end{llll}
   \begin{proof}
     If
     $x_{i_1}$ and
     $x_{i_2}$ are in the same connected component of
     $G_u$, then the statement follows from Lemma~%
\ref{equality}.

     Conversely, let
     $x_{i_1}$ and
     $x_{i_2}$ be in different connected components of
     $G_{u}$. We need to show that
     $u\neq 0$  in this case.

     Let
     $\pi_{X_u}$ be the projection of the algebra
     $M(X;G)$ onto the algebra
     $M(X_u,G_u)$. Since
     $\pi_{X_u}$ acts identically on the set
     $X_u$ we have
     $\pi_{X_u}(u)=u$ .

     Consider a set
     $Y$ such that it contains by one element from each connected component of
     $G_u$ and both
     $x_{i_1}$ and
     $x_{i_2}$ are in
     $Y$.

     Define the complete identification of connected components
     $\psi: M(X_u;G_u) \rightarrow M(Y)$ mapping each element in
     $X_u$ to the element in
     $Y$ from the same connected component of
     $G_u$.

     Applying
     $\psi$ to the monomial
     $u$ in
     $M(X_u;G_u)$, we obtain
     $$0=\psi([x_{i_1},x_{i_2},\dots, x_{i_k}])=[x_{i_1},x_{i_2},\psi(x_{i_3}),\dots,\psi(x_{i_k})].$$
     Since
     $x_{i_2}$ is the smallest letter of
     $u$ it is also the smallest letter in
     $\psi(u)=[x_{i_1},x_{i_2},\psi(x_{i_3}),\dots,\psi(x_{i_k})]$.
     By
(\ref{permut}), one can permute the letters
     $\psi(x_{i_3}),\dots, \psi(x_{i_k})$ in
     $\psi(u)$ in such a way that they are in the non-descending order in the obtained monomial.
     The obtained monomial is of the form
(\ref{basis}) and it is equal to
     $\psi(u)$ in
     $M(Y)$. Therefore, it cannot be equal to
     $0$ in this algebra. Consequently,
     $u\neq 0$ in
     $M(X;G)$ either.
   \end{proof}

   Let us consider the algebra
   $M(X;G)$. By
   $B'_{\overline{\delta}}(X)$ denote the set of all monomials of the form
(\ref{basis}) whose multidegrees are  equal to
   $\overline{\delta}$. Let us eliminate from
   $B'_{\overline{\delta}}(X)$ all monomials
   $u$ such that
   $u=0$ in
   $M(X;G)$ (by Lemma~%
\ref{zero}, we should exclude all monomials
   $u=[x_{i_1},x_{i_2},\dots, x_{i_k}]$ such that
   $x_{i_1}$ and
   $x_{i_2}$ are in the same connected component of the graph
   $G_u$). By
   $B''_{\overline{\delta}}(X;G)$ denote the obtained set.

   Given monomials
   $u_1=[x_{i_1},x_{i_2},\dots x_{i_k}]$  and
   $u_2=[x_{j_1},x_{j_2},\dots x_{j_k}]$ in
   $B''_{\overline{\delta}}(X;G)$ we write
   $u_1 \sim u_2$ if
   $x_{i_1}$ and
   $x_{j_1}$ are in the same connected component of
   $G_{u_1}=G_{u_2}$.
    By Lemma~%
\ref{equality}, it follows that if
   $u_1 \sim u_2$ then
   $u_1=u_2$  in
   $M(X;G)$. Besides, it follows from
(\ref{basis}) that
   $x_{i_2}=x_{j_2}$. The relation
   ``$\sim$'' is obviously  an equivalence relation. Consequently,
   there is the decomposition of
   $B''_{\overline{\delta}}(X;G)$ by the equivalence classes. Let us choose
   by one element from each equivalent class in such a way that the first letter
   of each chosen element is the the largest one among the first letters of all
   monomials in this equivalence class. Denote this set by
   $B_{\overline{\delta}}(X;G)$ (in particular, this set can be empty for
   some multidegrees).

   Let us set
   $$\mathfrak{B}(X;G)=\bigcup_{\overline{\delta}}B_{\overline{\delta}}(X;G),$$
   where the union is taken by all multidegrees.

   The following theorem holds.

   \begin{ttt}\label{baspcom}
     The set
     $\mathfrak{B}(X;G)$ is a basis of the partially commutative metabelian
     Lie algebra
     $M(X;G)$.
   \end{ttt}
   \begin{proof}
     Since all elements of the same equivalence class are equal to each
     other in
     $M(X;G)$ and since
     $\mathfrak{B}(X;G)$ contains elements from all equivalence
     classes of the set
     $\bigcup_{\overline{\delta}} B''_{\overline{\delta}}(X;G)$
     the algebra
     $M(X;G)$ is spanned by
     $\mathfrak{B}(X;G)$.

     We are left to show that the elements of
     $\mathfrak{B}(X;G)$ are linearly independent. Let
     \begin{equation} \label{lincomb}
       g=\sum_{j}\alpha_j u_j,
     \end{equation}
     where
     $u_j=[x_{j_1},x_{j_2},\dots x_{j_{k_j}}]\in \mathfrak{B}(X;G)$,
     $\alpha_j \in R$. Suppose that
     $g=0$ in
     $M(X;G)$. Without loss of generality (see Sec.~%
\ref{prelim}) we may assume that
     $g$ is a homogeneous polynomial. In particular, all
     $k_j$ are equal to each other. For this reason, we write
     $k$ instead of
     $k_j$. Note that the second letter in all monomials
     $u_j$ is same, namely, it is the smallest letter appearing in
     $u_j$. Let us denote this letter by
     $x$.

     By construction of
     $B_{\overline{\delta}}(X;G)$, the first letters
     of different monomials
     $u_j$ in decomposition
(\ref{lincomb}) belong to different connected components of
     $G_{g}$. Consider the set of the first letters of the monomials
     $u_j$ appearing in
(\ref{lincomb}). Extend this set to the set
     $Y$ such that
     $Y$ contains by one letter from each connected component of
     $G_{g}$ and
     $x\in Y$. It can be done because the first letter of a non-zero monomial
     $u_j$ cannot belong to the same connected component as
     $x$. So,
     $x$ is the least element of
     $Y$.

     Consider the projection
     $\pi_{X_g}: M(X;G) \rightarrow  M(X_{g};G_{g})$.
     Since
     $\pi_{X_g}$ acts identically on
     $X_g$ we have
     $\pi_{X_g}(g)=g$.

     Let
     $\psi: M(X_{g};G_{g})\rightarrow M(Y)$ be the complete identification of
     connected components mapping each element of
     $X_{g}$ to the element of
     $Y$ belonging to the same connected component of
     $G_{g}$.

     Applying the map
     $\psi$ to
(\ref{lincomb}) we get
     \begin{equation*}
       \begin{split}
          0 & =\sum_{j} \alpha_j \psi(u_j) \\
            & = \sum_{j} \alpha_j [x_{j_1}, \psi(x),\psi (x_{j_3})\dots, \psi (x_{j_k})].
       \end{split}
     \end{equation*}
     It was shown above that
     $\psi (x_{j_s})\in Y$. If it is necessary we can permute the letters
     $\psi(x_{i_3}),\dots, \psi(x_{i_k})$ in each
     $\psi(u_j)$ in such a way that they are in the non-descending order in the obtained monomial.
     We obtain
     $$\sum_{j}\alpha_j [x_{j_1},x,y_{j_3},\dots, y_{j_k}]=0$$
     in the free metabelian algebra
     $M(Y)$. Since
     $x_{j_1}$ and
     $x$ are in different connected components,
     $\psi (u_j)\neq 0$ in
     $M(Y)$. By Theorem~%
\ref{freemetabbas}, the set of monomials
     $\psi(u_j)$ is linearly independent in
     $M(Y)$. Consequently,
     $\alpha_j=0$ for all
     $j$. We have proved that the monomials
     $u_j$ are linearly independent in
     $M(X_{g};G_{g})$ so are they in
     $M(X;G)$. This completes the proof.
   \end{proof}

   \section{Annihilators and centralizers}\label{ann&cent}
   Let us start this section with proving some auxiliary statements about free metabelian Lie algebras.
   We need them to describe annihilators and centralizers of partially commutative Lie algebras.

   \begin{llll}\label{monombychar}
     Suppose
     $u=[x_{j_1},x_{j_2},\dots,x_{j_m}]\in \mathfrak{B}(X)$, where
     $m\geqslant 2$. Let
     $v$ be the largest monomial in the representation of
     $u.z$ as a linear combination of basis elements of
     $M(X)$. Then the
     first letter of
     $v$ is
     $x_{j_1}$ and the coefficient by this monomial is equal to
     $1$.
   \end{llll}
   \begin{proof}
     Let us consider the representation of
     $u.z$ as a linear combination of the elements of the form
(\ref{basis}). We need to consider two cases.\\[-1ex]

     \noindent
     1. Let
     $z\geqslant x_{j_2}$. In this case,
     \begin{equation*}
       \begin{split}
             v&=[x_{j_1},x_{j_2},x_{j_3},\dots, x_{j_m}].z\\
              &=[x_{j_1},x_{j_2},\dots,x_{j_l},z,x_{j_{l+1}}\dots, x_{j_m}],
       \end{split}
     \end{equation*}
     where
     $x_{j_l}\leqslant z < x_{j_{l+1}}$. We obtain a monomial of the form
(\ref{basis}) whose first letter is
     $x_{j_1}$.\\[-1ex]

     \noindent
     2. Let
     $z<x_{j_2}$. Then
     \begin{equation*}
       \begin{split}
             v&=[x_{j_1},x_{j_2},x_{j_3},\dots, x_{j_m}].z\\
              &=[x_{j_1},x_{j_2},z,x_{j_3}\dots, x_{j_m}]\\
              &=[x_{j_1},z,x_{j_2},x_{j_3}\dots, x_{j_m}]+
                \bigl[\dots[[x_{j_1},[x_{j_2},z]],x_{j_3}],\dots, x_{j_m}\bigr]\\
              &=[x_{j_1},z,x_{j_2},x_{j_3}\dots, x_{j_m}]-
                [x_{j_2},z,x_{j_1},x_{j_3}\dots, x_{j_m}]
       \end{split}
     \end{equation*}
     We obtain a linear combination of two monomials of the form
(\ref{basis}). Since
     $x_{j_2}$ is the smallest letter of
     $u$ the monomial with the first letter
     $x_{j_1}$ is larger than the other one and the coefficient by this monomial
     is equal to
     $1$.
   \end{proof}

   \begin{llll}\label{multybychar}
     Let
     $c=\alpha\overline{c}+\Delta c$ be the representation of
     $c\in M'(X)$ as a linear combination of basis elements with non-zero coefficients and let
     $\overline{c}=[y,x,x_{i_3},\dots,x_{i_m}]$. Then for any
     $z \in X$ the largest monomial in the decomposition of
     $c.z$ as a linear combination of basis elements begins with the letter
     $y$ and the coefficient by this monomial is equal to
     $\alpha$.
   \end{llll}

   \begin{proof}
     Let
     $c=\alpha\overline{c}+\Delta c$ be a homogeneous Lie monomial
     consisting of monomials of the form
(\ref{basis}) and
     $d$ a monomial appearing in
     $\Delta c$. Since
     $\overline{c}$ is the greatest monomial of
     $c$ we have either
     $\mdeg(\overline{c})>\mdeg(d)$ or
     $\mdeg(\overline{c})=\mdeg(d)$ but
     $\overline{c}$ is greater than
     $d$ lexicographically. Let us represent
     $\alpha(\overline{c}.z)$ and
     $(\Delta c).z$ as sums of elements of the form
(\ref{basis}). By Lemma
  \ref{monombychar}, the largest monomial of the decomposition of
     $\alpha(\overline{c}.z)$ begins with
     $y$ and the coefficient by this monomial is
     $\alpha$. On the other hand, the monomials
     in the decomposition of
     $(\Delta c).z$ are  less than
     $y$ with respect to the standard order. Let
     $b$ be the
     largest monomial in the decomposition of
     $c.z$ to a linear combination of elements of the form
(\ref{basis}). We have
     $b$ begins with
     $y$ and the coefficient by this monomial is equal to
     $\alpha$.
   \end{proof}

   The following result has been proved in
\cite{DKR06} in the case of a free metabelian Lie algebra under
   a field. In his paper, we give the proof of this statement for a free metabelian Lie
   algebra over an arbitrary integral domain.

   \begin{ttt}\label{torsfree}
     The derived subalgebra
     $M'(X)$ of
     $M(X)$ is a torsion-free
     $R[X]$-module.
   \end{ttt}
   \begin{proof}
     We need to prove that
     $c.f\neq 0$ whenever
     $c\in M'(X)\backslash \{0\}$ and
     $f\in R[X]\backslash\{0\}$.

     Note that
     $u.a\neq 0$ implies
     \begin{equation}\label{muldeg}
       \mdeg(u.a)= \mdeg(u)+\mdeg(a)
     \end{equation}
     for any Lie monomial
     $u$ and any associative monomial
     $a$. So, without loss of
     generality it can be assumed that the decomposition of
     $c$ to a linear combination of monomials of the form
(\ref{basis}) is homogeneous by all generators and
     $f$ is a monomial. Indeed, we can write
     $c=c_{\overline{\gamma}}+\sum_{\overline{\delta}} c_{\overline{\delta}}$, where
     $c_{\overline{\delta}}$ is a homogeneous Lie polynomial with the multidegree
     $\overline{\delta}$ and
     $\overline{\gamma}$ is the largest multidegree of the monomials in the decomposition of
     $c$. Let also
     $f=\alpha\overline{f}+\Delta f$, where
     $\overline{f}$ is the monomial of the largest degree appearing in
     $f$. It is sufficient to show that
     $c_{\overline{\gamma}}.\overline{f}\neq 0$. Indeed, by
(\ref{muldeg}), the multidegrees of all other summands in the
     decomposition of
     $c.f$ are less than the multidegree of
     $c_{\overline{\gamma}}.\overline{f}$. Therefore, this monomial cannot cancel with the others.

     So, suppose
     $c$ is represented as a linear combination of monomials of the form
(\ref{basis}),
     $\overline{c}=[y,x,x_{i_3},\dots,x_{i_m}]$ is the largest monomial in
     this linear combination, and
     $f$ is an associative monomial of the multidegree
     $\overline{\varepsilon}$. By
     $|\overline{\varepsilon}|$ denote the sum of coordinates of
     $\overline{\varepsilon}$. Let us proceed by induction on
     $|\overline{\varepsilon}|$.

      If
      $|\overline{\varepsilon}|=1$ then
      $f=z$ for some
      $z\in X$. Let
      $b$ be the largest monomial in the representation of
      $c_{\overline{\gamma}}.z$ as a linear combination of
      elements of the form
(\ref{basis}). By Lemma~%
 \ref{multybychar},
      $b$ begins with
      $y$. The coefficient by this monomial is equal to the
      coefficient by
      $\overline{c}$. In particular, this coefficient is not equal
      to zero. Therefore,
      $c_{\overline{\gamma}}.z\neq 0$. By
(\ref{muldeg}), this monomial cannot cancel with monomials in the
      decompositions of
      $c_{\overline{\delta}}.z$. Consequently,
      $c.z\neq 0$.

      Suppose that the statement holds for all
      $\overline{\varepsilon}$ such that
      $|\overline{\varepsilon}|<r$. Let
      $|\overline{\varepsilon}|=r>0$ and let
      $f$ be a monomial of the multidegree
      $\overline{\varepsilon}$. For some generator
      $x_i$, we can write
      $f=\beta x_i f_0$, where
      $\beta\in R\backslash\{0\}$ and
      $|\mdeg{f_0}|=r-1$. We have
      $c.f= c.(\beta x_if_0)=\beta (c.x_i).(f_0)$. As above, we
      get
      $(c.x_i)\neq 0$. Consequently, by the inductive hypothesis
      we obtain
      $(c.x_i).(f_0)\neq 0$. The proof is complete.
   \end{proof}

   The following statement on commutative associative polynomials seems to be
   well-known, but anyway, we give its proof here.

   \begin{llll}\label{asspoly}
     Let
     $R$ be an infinite integral domain and let
     $f\in R[x_1,\dots,x_n]$ be a non-zero polynomial over
     $R$ in a finite number of indeterminates. Then there exist
     $r_1,\dots,r_n\in R$ such that
     $f(r_1,\dots, r_n)\neq 0$.
   \end{llll}
   \begin{proof}
     Proceed by induction on the number of indeterminates. Let
     $f(x)=\alpha_l x^l+\dots+ \alpha_1 x+\alpha_0$ be a non-zero
     polynomial in one indeterminate. Let us show that this polynomial cannot have
     more than
     $l$ roots. If
     $y_1$ is a root then there is the representation
     $f(x)=(x-y_1)f_1(x)$, where the degree of
     $f_1(x)$ is equal to
     $l-1$. Similarly, one can find a root of
     $f_1(x)$, denote this root by
     $y_2$, and obtain a similar decomposition for
     $f_1(x)$ and so on. Finally, we obtain the representation
     \begin{equation}\label{multass}
       f(x)=(x-y_1)\dots (x-y_s)g(x),
     \end{equation}
     where
     $g(x)$ has no roots. Obviously,
     $f(x)$ has no roots except
     $y_1,\dots y_s$. Indeed, if
     $z$ is a root then substitution of
     $z$ for
     $x$ in
(\ref{multass}) gives
     $f(z)=(z-y_1)\dots (z-y_s)g(z)$. Since
     $R$ is an integral domain, at least one of the multiples is equal to
     $0$ and
     $g(z)\neq 0$ because
     $g(x)$ has no roots. Therefore,
     $z=y_i$ for some
     $i$ such that
     $1\leqslant i \leqslant s$.

     Since
     $R$ is infinite, one can choose
     $r$ not equal to any root of
     $f(x)$. Then
     $f(r)\neq 0$.

     Suppose that the statement of the lemma holds for all
     polynomials in less than
     $k$ indeterminates. Consider a polynomial
     $f$ in
     $k$ indeterminates. We have
     \begin{equation}\label{multypoly}
       f=\sum_{i=0}^l x_k^i f_i,
     \end{equation}
     where
     $f_i\in R[x_1,\dots x_{k-1}]$. Since
     $f\neq 0$ there exists
     $i$ such that
     $f_i\neq 0$. By the inductive assumption, there
     are
     $r_1,\dots, r_{k-1}$ such that
     $f_i(r_1,\dots,r_{k-1})\neq 0$ for some
     $i$. Substituting these values for the corresponding
     indeterminates in
(\ref{multypoly}) we get a non-zero polynomial in the
     indeterminate
     $x_k$ only. Consequently, there exists
     $r_k\in R$ such that substituting this element for
     $x_k$ we obtain a non-zero element of
     $R$. So, the substitution
     $r_1,\dots r_{k}$ for the corresponding indeterminates in
(\ref{multypoly}) gives a non-zero element of
     $R$. This concludes the proof.
   \end{proof}
   \begin{llll}\label{dirsumab}
    Let
    $G$ be a graph such that each its connected component is a complete graph and
    $M(X;G)$ the partially commutative Lie algebra with the defining graph
    $G$. If
    $[x,y].f=0$ for some vertices
    $x$ and
    $y$ and for some polynomial
    $f\in R[X]$ then either
    $[x,y]=0$ (i.e
    $x$ and
    $y$ belong to the same connected component of
    $G$) or
    $f=0$.
  \end{llll}
  \begin{proof}
    If
    $G$ is a connected graph then
    $M(X;G)$ is just a free abelian Lie algebra. Therefore,
    $[x,y]=0$ for any
    $x,y\in X$.

    If
    $G$ has at least two connected components then we proceed by induction
    on the number of generators in
    $M(X;G)$. Suppose
    $|X|=2$ and
    $G$ has two connected components (it means that
    $G$ is totally disconnected). Then we have
    $M(X;G)$ is the free metabelian Lie algebra with two generators and the
    statement follows from Theorem~%
\ref{torsfree}.

    Let the statement hold for all graphs having less than
    $k$ vertices and satisfying the conditions of the lemma. Consider a graph
    $G$ such that
    $|X|=k$. If this graph is totally disconnected then
    $M(X;G)$ is a free metabelian algebra and the statement follows from Theorem~%
 \ref{torsfree}.

    Suppose that
    $G$ contains some edges and
    $[x,y].f=0$, where
    $f\in R [X]$ and
    $x$ and
    $y$ are in different connected components. Consider an arbitrary connected component
    $\Gamma$ of
    $G$ such that
    $\Gamma$ has at least two vertices. Let
    $Z=\{z,z_1,\dots z_{m}\}$ be the set of all vertices in this component.
    If
    $x\in \Gamma$ (or
    $y\in \Gamma$) then let us set
    $z=x$
    ($z=y$, respectively). By
    $Y$ denote the set obtained from
    $X$ by deleting vertices
    $z_1,\dots, z_{m}$.

    Let
    $\varphi: M(X;G)\rightarrow M(Y;G_Y)$ be an identification on the connected component
    $\Gamma$ defined as follows:
    \begin{equation}\label{glueconst}
      \varphi(x)=
      \begin{cases}
         \alpha_i z & \text{if } x=z_i\\
         x & \text{else}
      \end{cases}.
    \end{equation}

    We have the following representation of
    $f$:
    $$f=\sum_j \gamma_j f_j z^{a_j}z_1^{a_{j,1}}\dots z_m^{a_{j,m}},$$
    where
    $f_j$ are monomials in
    $R[X\backslash Z]$. We obtain
    $$\varphi(f)=\sum_j \gamma_i\alpha_1^{a_{j,1}}\dots \alpha_m^{a_{j,m}} f_j z^{a_i+a_{j,1}+\dots+a_{j,m}}.$$
    Combining like terms by degrees of the elements in
    $Y$ we obtain a sum such that each its coefficient is a polynomial in
    $\alpha_1,\dots, \alpha_m$. Moreover, since
    $f\neq 0$ at least one of the obtained coefficients is not equal to zero identically. Therefore, by Lemma~%
\ref{asspoly}, one can find
    $\alpha_1,\dots,\alpha_n$ such that
    $\varphi(f)\neq 0$. From
(\ref{glueconst}), it follows that
    $\varphi$ cannot map vertices from different connected
    components to one component. Consequently,
    $\varphi([x,y])=[x,y]\neq 0$ . Note that
    $M(Y;G_Y)$ satisfies the conditions of the lemma and the number of vertices in
    $G_Y$ is less than
    $k$. Therefore, by the inductive hypothesis,
    $[x,y].\varphi(f)\neq 0$ in
    $M(Y;G_Y)$. It means that
    $[x,y].f\neq 0$ in
    $M(X;G)$. The proof is complete.
  \end{proof}

  \begin{ddd}
    Let
    $A$ be an associative commutative ring and let
    $M$ be an
    $A$-module. An \emph{annihilator} of a non-zero element
    $m\in M$ is the ideal
    $A$ consisting of all elements
    $a\in A$ such that
    $m\cdot a=0$.
  \end{ddd}
  Since for any partially commutative algebra
  $M(X;G)$ its derived subalgebra
  $M'(X;G)$ is an
  $R[X]$-module, the annihilator of any element
  $g\in M'(X;G)$ in
  $R[X]$ is defined.

  Let us move on to the description of annihilators in
  $M'(X;G)$. Define the ideal
  $I^G_{i,j}$ of
   $R[X]$ as follows. If
   $x_i$ and
   $x_j$ belong to different connected components in
   $G$ then put
   $I^G_{i,j}=0$. Let
   $x_i$ and
   $x_j$ be in the same connected component. For each path
   $(x_i,y_1,y_2,\dots, y_s,x_j)$ connecting these vertices in
   $G$ consider the monomial
   $y_1y_2\dots y_s$. Define
   $I^G_{i,j}$ as the ideal generated by all such monomials. Obviously,
   it is sufficient to consider the monomials corresponding to the simple paths only.
   \begin{ttt}\label{ann}
     Let
     $M(X;G)$ be the partially commutative metabelian Lie algebra
     with the set of generators
     $X$ and the defining graph
     $G$ and let
     $x_i,x_j\in X$. If
     $x_i$ and
     $x_j$ are not adjacent then the annihilator of
     $[x_i,x_j]$ is equal to
     $I^G_{i,j}$.
   \end{ttt}
   \begin{proof}
     First of all, let us show that if
     $f\in I^G_{i,j}$, then
     $[x_i,x_j].f=0$. It is sufficient to prove this statement for the elements generating
     $I^G_{i,j}$.

     Let us show that the statement is true for the paths having one intermediate vertex,
     i.e in the case
     $s=1$. We have
     $[x_i,x_j].y_1=[[x_i,x_j],y_1]=[[x_i,y_1],x_j]+[x_i,[x_j,y_1]]=0$.

     Let the statement hold for all simple paths having
     $s<k$ intermediate vertices. Suppose that the path
     $(x_i,y_1,\dots, y_k,x_j)$ has
     $k$ intermediate vertices. We have
     \begin{equation}\label{decom}
       [x_i,x_j].(y_1y_2\dots y_k)=\sum [x_i.(y_{i_1}\dots y_{i_t}),x_j.(y_{j_1}\dots y_{j_{k-t}})],
     \end{equation}
     where the sum is taken by all ordered partitions of the set
     $\{y_1,y_2,\dots,y_k\}$ by two subsets
     $\{y_{i_1}\dots y_{i_t}\}$ and
     $\{y_{j_1}\dots y_{j_{k-t}}\}$. Since
     $M(X;G)$ is metabelian it follows from
(\ref{decom}) that:
     \begin{equation*}
       \begin{split}
          [x_i,x_j].(y_1y_2\dots y_k)&=[x_i.(y_1\dots y_k),x_j]+[x_i,x_j.(y_1\dots y_k)]\\
                                     &= [[x_i,y_1].(y_2\dots y_k),x_j]+
                                     [x_i,[x_j,y_1].(y_2\dots y_k)].
       \end{split}
     \end{equation*}
     Since
     $x_i$ are
     $y_1$ adjacent in
     $G$ the first summand is equal to
     $0$. The second summand contains the multiple
     $[x_j,y_1].(y_2\dots y_k)$. Considering the path
     $(y_1,y_2,\dots, x_j)$ we obtain that this multiple is equal to zero by the induction hypothesis.
     Consequently,
     $[x_i,x_j].(y_1y_2\dots y_k)=0$.

     So, we have shown that
     $I^G_{i,j}$ is a subset of the annihilator of
     $[x_i,x_j]$. Let us prove the inverse statement. Without loss of generality, the annihilator  of
     $[x_1,x_2]$ can be considered. There are two cases:\\[-1ex]

     \noindent
     1. Suppose that
     $x_1$ and
     $x_2$ belong to different connected components of
     $G$. Let us add edges to the graph
     $G$ in such a way that any two vertices belonging to the same
     connected component of
     $G$ would be connected. Denote this graph by
     $\overline{G}$. So, connected components of
     $G$ and
     $\overline{G}$ consist of the same vertices and each connected component of
     $\overline{G}$ is a complete graph.

     Consider the identical simplification
     $\varphi: M(X,G)\rightarrow M(X,\overline{G})$. If
     $[x_1,x_2].f=0$ in
     $M(X,G)$, then
     $\varphi([x_1,x_2].f)=0$. Therefore,
     $[x_1,x_2].f=0$ in
     $M(X,\overline{G})$. So, by Lemma~%
\ref{dirsumab}, we have
     $f=0$.\\

     \noindent
     2. If
     $x_1$ and
     $x_2$ are not adjacent and belong to the same connected component, then we proceed by induction on
     $|X|$. Let
     $|X|=3$. Then
     $G$ has exactly two edges, namely
     $\{x_1,x_3\}$ and
     $\{x_2,x_3\}$. Indeed, otherwise, either
     $x_1$ and
     $x_2$ are adjacent or they belong to different connected components.

     Consider the projection
     $\pi_{Y}: M(X;G) \rightarrow M(Y;G_Y)$, where
     $Y=\{x_1,x_2\}$. Let
     $[x_1,x_2].f=0$. We have
     $\pi_Y ([x_1,x_2].f)=[x_1,x_2].\pi_Y(f)=0$. Since
     $G_Y$ is totally disconnected
     $M(Y;G_Y)=M(Y)$. Therefore, by Theorem~%
\ref{torsfree},
     $\pi_Y(f)=0$. Consequently,
     $f=x_3\widehat{f}$.

     Suppose that the statement holds for
     $|X|<n$. Consider an algebra
     $M(X;G)$, where
     $X$ contains
     $n$ vertices.  Let
     $(x_1,y_1,y_2,\dots, y_s,x_2)$ be a simple path connecting
     $x_1$ and
     $x_2$.

     By
     $X_i$ denote the set
     $X\backslash\{y_i\}$, where
     $i=1,2,\dots, s$. Let
     $G_i$ be the graph generated by
     $X_i$.

     Consider the set of projections
     $\pi_{X_i}: M(X;G)\rightarrow M(X_i;G_i)$, where
     $i=1,2,\dots,s $. Note that
     $\pi_{X_i}(x_1)=x_1$ and
     $\pi_{X_i}(x_2)=x_2$ for all
     $i$ because the path
     $(x_1,y_1,y_2,\dots, y_s,x_2)$ is simple. Therefore, if
     $[x_1,x_2].f=0$ in
     $M(X;G)$ then
     $\pi_{X_1}([x_1,x_2].f)=[x_1,x_2].\pi_{X_1}(f)=0$ in
     $M(X_1;G_1)$. Let us remark that
     $M(X_1;G_1)$ is a subalgebra of
     $M(X;G)$ and
     $|X_1|=n-1$. Consequently, by the inductive hypothesis,
     $\pi_{X_1}(f)$ is in the ideal
     $I^{G_1}_{1,2}$ of the algebra
     $M(X_1;G_1)$. In particular,
     $\pi_{X_1}(f)$ does not contain
     $y_1$. Since
     $\pi_{X_1}$ acts identically on all generators but
     $y_1$ there is the representation
     $f=\pi_{X_1}(f)+f_1$, where
     $f_1\in R[X]$ is such that
     $\pi_{X_1}(f_1)=0$.  Consequently,
     $f_1=y_1 \widehat{f}_1$. It is clear that
     $I^{G_1}_{1,2} \subseteq I^G_{1,2}$. Therefore,
     $\pi_{X_1}(f)\in I^G_{1,2}$.

     Suppose that for some
     $i<s$ there is the representation
     \begin{equation}\label{sumid}
       f=f_i+\sum_{j=1}^i \widetilde{f}_{j},
     \end{equation}
     where
     $\widetilde{f}_{j}\in I^{G_j}_{1,2}$, and
     $f_i=y_1y_2\dots y_i \widehat{f}_i$. Obviously,
     $I^{G_j}_{1,2} \subseteq I^G_{1,2}$ for
     $j=1,2\dots, i$. Thus, we have
     $0=[x_1,x_2].f=[x_1,x_2].f_i+
       \sum_{j=1}^i [x_1,x_2].\widetilde{f}_{j}=
       [x_1,x_2].f_i$. Therefore,
     $0=\varphi_{i+1}([x_1,x_2].f_i)=[x_1,x_2].\pi_{X_{i+1}}(f_i)$.
     So, by the induction hypothesis,
     $\pi_{X_{i+1}}(f_i)$ belongs to the ideal
     $I^{G_{i+1}}_{1,2}$ of the algebra
     $M(X_{i+1};G_{i+1})$. As above, we obtain
     $f_i=f_{i+1}+\widetilde{f}_{i+1}$, where
     $\widetilde{f}_{i+1}=\pi_{X_{i+1}}(f_i)\in I^{G_{i+1}}_{1,2} \subseteq I^G_{1,2}$
     and
     $f_{i+1}$ is such that
     $\pi_{X_{i+1}}(f_{i+1})=0$. Therefore,
     $f_{i+1}=y_{i+1}f'_{i+1}$. Since each monomial of
     $f_i$ is of the form
     $g=y_1y_2\dots y_i g'$ we obtain
     $f_{i+1}=y_1y_2\dots y_{i+1}\widehat{f}_{i+1}$. Consequently,
     there is the following decomposition:
     $f=f_{i+1}+\sum_{j=1}^{i+1} \widetilde{f}_{j}$. So,
     decomposition
(\ref{sumid}) holds for any
     $i$.

     Setting
     $i=s$ in
(\ref{sumid}) we obtain
     $f=y_1y_2\dots y_s\widehat{f}_s+\sum_{j=1}^{s} \widetilde{f}_{j}$.
     As above,
     $\sum_{j=1}^{s} \widetilde{f}_{j}\in I^G_{1,2}$. Since
     $y_1y_2\dots y_s$ is a generating element of
     $I^G_{1,2}$, the monomial
     $y_1y_2\dots y_s\widehat{f}_s$ is also in this ideal.  So, the proof is complete.
   \end{proof}

   Before we move on to a description of centralizers in partially commutative metabelian Lie algebras
   we prove one more auxiliary assertion.

   Let
   $X_{c,z}=X_{c}\cup\{z\}$ for any
   $z\in X\backslash X_c$, and let
   $G_{c,z}$ be the subgraph of
   $G$ generated by
   $X_{c,z}$.

   Remind that
   $z$ is called a \emph{cutpoint} of
   $G$ if its deletion increases the number of connected components.

   \begin{llll}\label{monombycharpc}
     Let
     $c\in M'(X,G)$ be a non-zero element that can be written as a homogeneous
     linear combination of basis elements. If
     $c.z=0$ for some
     $z\in X$ then
     $z$ does not appear in
     $c$ and it is a cutpoint of
     $G_{c,z}$.
   \end{llll}
   \begin{proof}
     Let us reorder vertices of
     $X$ in such a way that
     $z$ would be not less than the vertices appearing in
     $c$. If
     $c$ contains
     $z$ or if
     $z$ is greater than the smallest letter appearing in
     $c$ then we can use the initial order on
     $X$. Otherwise, we can just permute
     $z$ and the smallest letter of
     $c$ in the initial linear order of
     $X$. Note that if
     $c$ admits a representation as a homogeneous linear
     combination of basis elements under some order on
     $X$ then it can be rewritten in such form under any other
     order on
     $X$. Moreover, for all orders on
     $X$ the multidegrees of the monomials in the corresponding
     representations are same.

     By Lemma
\ref{zero},
     $G_c$ has at least two connected components. Suppose that
     $\Gamma_{0},\Gamma_{2},\dots, \Gamma_{s}$ are connected components
     of
     $G_c$ (%
     $s\geqslant 1$). Let
     $y_0,y_1,\dots, y_{s}$ be the largest vertices of these
     components such that
     $y_i$ lies in
     $\Gamma_{i}$ for each
     $i=0, \dots, s$. We may assume that
     $y_0<y_1<y_2<\dots<y_s$.
     By
     $x$ denote the minimal vertex appearing in
     $c$. Let
     $x$ belong to
     $\Gamma_{0}$.

     We have:
     \begin{equation} \label{basedecom}
       c=\sum_{i=1}^s \beta_i [y_i,x,x_{i_3},\dots, x_{i_m}],
     \end{equation}
     where all summands have the same multidegree and some
     $\beta_i$ are not equal to
     $0$. Let
     $r\leqslant s$ be the largest number such that
     $\beta_r\neq 0$. Consider
     $c.z$ as an element of
     $M(X)$. Since
     $z\geqslant x$,  the proofs of Lemma~
\ref{monombychar} (case 1) and Lemma~%
 \ref{multybychar} imply that the largest monomial of
     $c.z$ can be written as follows:
     $\overline{c.z}=\overline{c}.z=[y_r,x,x'_{r_3},\dots, x'_{r_{m+1}}]$.

     Let
     $c.z=0$ in
     $M(X;G)$. If
     $\overline{c}.z=0$ in
     $M(X;G)$, then Lemma~%
\ref{zero} implies that
     $y_r$ and
     $x$ belong to the same connected component of
     $G_{c,z}$, but in different connected components of
     $G_c$. Consequently,
     $z$ does not appear in
     $c$ and there are vertices in
     $\Gamma_{0}$ and
     $\Gamma_{r}$ adjacent to
     $z$. It means that
     $z$ is a cutpoint of
     $G_{c,z}$.

     If
     $\overline{c}.z\neq 0$ in
     $M(X;G)$, then the monomial
     $\overline{c}.z$ is equal to
     $c'=[y_j,x,\dots, x'_{j_{m+1}}]$ for some
     $j\neq r$ in this algebra and
     $\mdeg(\overline{c})=\mdeg(c')$. But it is possible only if
     $y_j$ and
     $y_r$ are in the same connected component of
     $G_{c,z}$. Therefore,
     $z$ does not appear in
     $c$ and it is a cutpoint of
     $G_{c,z}$.
   \end{proof}

   Let
   $G$ be an arbitrary graph with the set of vertices
   $X$ and let
   $u$ be an arbitrary element of
   $M(X;G)$. By
   $C_{G}(u)$ denote the centralizer of
   $u\in X$ in the algebra
   $M(X;G)$. So,
   $C_{G}(u)$ is the set of all elements
   $v\in M(X;G)$ such that
   $[v,u]=0$. Let us also introduce the notation
   $\mathcal{C}_G(x)$ for
   $C_{G}(x)\cap M'(X;G)$. We are going to describe the centralizers
   of the elements in
   $X$ and their linear combinations.  We need this description in the proof of
   Theorem~%
\ref{main} in Sec.~%
 \ref{univequiv}. The following theorem holds.

  \begin{ttt}\label{centrgen}
    Let
    $G$ be a graph with the set of vertices
    $X=\{x_1,x_2,\dots,x_n\}$.\\
    1. If
       $x_n$ is an isolated vertex in
       $G$ then
       $C_G(x_n)$ consists of the elements
       $v$ of the form
       \begin{equation} \label{centralizeriso}
         v=\alpha_n x_n,
       \end{equation}
       where
       $\alpha_n \in R$.\\
    2. If the degree of
       $x_n$ is equal to
       $1$ in
       $G$ (say, it is adjacent to
       $x_{n-1}$) then
       $C_G(x_n)$ consists of all elements
       $v$ of the form
       \begin{equation} \label{centralizerdan}
         v=\alpha_{n-1}x_{n-1}+\alpha_n x_n,
       \end{equation}
       where
       $\alpha_{n-1},\alpha_n\in R$.\\
    3. If
       $x_n$ is adjacent to
       $x_{r+1},\dots, x_{n-1}$ in
       $G$ ($r \leqslant n-3$), then
       $C_G(x_n)$ consists of all elements
       $v$ of the form
       \begin{equation}\label{centralizer}
         v=\sum_{k=r+1}^n \alpha_k x_k +\sum_{r+1\leqslant i<j\leqslant n-1}[x_i,x_j].f_{ij},
       \end{equation}
       where
       $\alpha_k \in R$,
       $f_{ij}\in R[X\backslash \{x_n\}]$.
  \end{ttt}
  \begin{proof}
    Clearly, all elements of the form
(\ref{centralizeriso}) or
 (\ref{centralizerdan}) (depending on the number of vertices adjacent to
    $x_n$) in
    $M(X;G)$ are in
    $C_G(x_n)$. It is also easy to see that if
    $x_n$ is adjacent to
    $x_{r+1},\dots, x_{n-1}$ then any element of the form
    $v=\sum_{k=r+1}^n \alpha_k x_k$ belongs to the centralizer of
    $x_n$. So, we are left to show that the elements of the form
    $\sum_{r+1\leqslant i<j\leqslant n-1}[x_i,x_j].f_{ij}$ (see case 3)
    are in
    $C_G(x_n)$. We need to prove that
    $$\sum_{r+1\leqslant i<j\leqslant n-1}[x_i,x_j].f_{ij}x_n=0.$$
    It follows from Lemma~%
\ref{zero}. Indeed, let us consider the element of the
    form
    $w=[x_i,x_j].f_{ij}x_n$. Since
    $x_i$ and
    $x_j$ are adjacent to
    $x_n$, the vertices
    $x_i$ and
    $x_j$ are in the same connected component of
    $G_w$ and we are done.

    Conversely, suppose that
    $c\in C_G(x_n)$. Let us write this element as a linear combination of basis
    monomials. As we noticed in Sec.~%
\ref{prelim} if a Lie polynomial is equal to zero in
    $M(X;G)$ then all homogeneous components (in the sense of
    multidegree) of this polynomial are equal to zero in
    $M(X;G)$. So, without loss of generality it can be assumed that all basis monomials of
    $c$ have the same multidegree. Consider two cases.\\[-1ex]

    \noindent
    1. If a homogeneous element
    $c\in C_G(x_n)$ is a linear combination of monomials of the length
    $1$, then
    $c=\alpha_i x_i$ for some
    $i$. Consequently,
    $0=[c,x_n]=\alpha_i [x_i,x_n]$. So, it is easy to see that either
    $i=n$ or
    $x_i$ is adjacent to
    $x_n$.\\[-1ex]

    \noindent
    2. Suppose that
    $c\in \mathcal{C}_G(x_n)$. Let us represent
    $c$ as a linear combination of basis elements of the algebra
    $M(X;G)$. Consider the graph
    $G_c$. Let this graph have
    $s+1$ connected components
    ($s\geqslant 1$). By
    $y_0,y_1,y_2,\dots,y_{s}$ denote the largest vertices of the connected
    components as we did in the proof of Lemma~%
\ref{monombycharpc}. Let
    $x$ be the smallest letter appearing in
    $c$. We can suppose that
    $x$ is in the same connected component as
    $y_{0}$. Since
    $[c,x_n]=0$,  Lemma~%
\ref{monombycharpc} implies, that
    $x_n$ is a cutpoint of
    $G_{c,x_n}$.

    Clearly, degree of
    $x_n$ in
    $G$ is not less than
    $2$.  Indeed,
    $x_n$ should  be connected with at least two connected components of
    $G_c$. Therefore, its degree in
    $G_{c,x_n}$ should be greater than
    $1$. So, if the degree of
    $x_n$ in
    $G$ is not greater than
    $1$ then
    $\mathcal{C}(x_n)=0$ and parts 1 and 2 hold.

    Thus,
    $x_n$ is a cutpoint of
    $G_{c,x_n}$. We can assume without loss of generality that one of the connected components contains
    the vertices
    $y_1, \dots y_t$ for
    $t\leqslant s$, and perhaps
    $y_0$. Moreover, one can suppose that
    $y_1,y_2, \dots , y_t$ are adjacent to
    $x_n$, i.e.
    \begin{equation}\label{subset}
       \{y_1,\dots, y_t\}\subseteq \{x_{r+1},\dots,x_n\}.
    \end{equation}
    We can also assume that if
    $y_0$ is in the same connected component of
    $G_{c,x_n}$ as the vertices
    $y_1,\dots,y_t$ then
    $y_0$ is also adjacent to
    $x$.

    Since the representation of
    $c$ in the form
(\ref{basedecom}) holds we obtain
    \begin{equation}\label{productxn}
      \begin{split}\\
        [c,x_n] &=\sum_{i=1}^s \beta_i[y_i,x,x_{i_3},\dots,x_{i_m},x_n]\\
          &= \sum_{i=1}^t \beta_i [x_n,x,x_{i_3},\dots,x_{i_m},y_i]+
             \sum_{j=t+1}^s \beta_j[y_j,x,x_{j_3},\dots,x_{j_m},x_n].
      \end{split}
    \end{equation}
    By Theorem~%
\ref{baspcom}, the monomials
    $[y_j,x,x_{j_3},\dots,x_{j_m},x_n]$ are linearly independent for
    $t+1 \leqslant j\leqslant s$. By the same theorem,  we see
    that the first sum of the right hand side
(\ref{productxn}) is equal to a multiple of a monomial of the
    corresponding multidegree beginning with
    $x_n$ which is either equal to zero in
    $M(X;G)$ or a basis element. In the latter case, this element is linearly independent from the elements of the
    second sum. Therefore, in both cases
    $\beta_j=0$ if
    $t+1 \leqslant j\leqslant s$. By
(\ref{basedecom}) we obtain
    \begin{equation}\label{productsimp}
        c = \sum_{i=1}^t \beta_i [y_i,x,x_{i_3},\dots,x_{i_m}].
    \end{equation}
    Consequently,
    \begin{equation}\label{productxnsimp}
        [c,x_n] = \sum_{i=1}^t \beta_i [y_i,x,x_{i_3},\dots,x_{i_m},x_n].
    \end{equation}

    In the former case, we obtain
    $x_n$ is in the same connected component of
    $G_{c,x_n}$ as
    $x$ and
    $y_0$. In this case, by Lemma~%
\ref{zero} we have
    $[y_i,x,x_{i_3},\dots,x_{i_m},x_n]=0$ for
    $i\leqslant t$. On the other hand,
    Lemma~
\ref{equality} implies that
    $[y_i,x,x_{i_3},\dots,x_{i_m},x_n]=[y_i,y_0,x'_{j_3},\dots,x'_{i_m},x'_{i_{m+1}}]$ and
    $y_0$ and
    $y_i$ are adjacent to
    $x_n$. So, as follows from
(\ref{subset}),
   $c$ can be represented in the form
 (\ref{centralizer}).

    In the latter case, we have
    $y_0$ and
    $x_n$ are in different connected components of
    $G_{c,x_n}$. By
(\ref{productxnsimp}), we obtain the following equation in
    $M(X;G)$:
    \begin{equation*}
        0=[c,x_n] = \sum_{i=1}^t \beta_i [x_n,x,x_{i_3},\dots,x_{i_m},y_i].
    \end{equation*}
    Let us arrange the last
    $m-1$ multiples in each summand of the right-hand side of this equation in the increasing
    order. We obtain a linear combination of
    $t$ equal monomials which are not equal to zero in
    $M(X;G)$. Since
    $[c,x_{n}]=0$ we obtain
    $\sum_{i=1}^t \beta_i=0$.
    So, the following equation holds
    \begin{equation}\label{sumofdif}
      c=\sum_{i=1}^{t-1}\beta_i([y_i,x,x_{i_3},\dots,x_{i_m}]-[y_t,x,x_{s_3},\dots,x_{s_m}]).
    \end{equation}

    For each summand of the right-hand side of
(\ref{sumofdif}), we can apply the Lie algebra identities and
    rearrange  the last
    $m-2$ multiples. So, we obtain
    \begin{equation}\label{difftoone}
      \begin{split}
        &[y_i,x,x_{i_3},\dots,x_{i_m}]-[y_t,x,x_{t_3},\dots,x_{t_m}]\\
        =& [y_i,x,y_t,x'_4\dots,x'_m]-[y_t,x,y_i,x'_4\dots,x'_m]\\
        =& [y_i,x,y_t,x'_4\dots,x'_m]+[x,y_t,y_i,x'_4\dots,x'_m]\\
        =& -[y_t,y_i,x,x'_4\dots,x'_m]\\
        =&[y_i,y_t,x,x'_4\dots,x'_m]
      \end{split}
    \end{equation}
     By
(\ref{sumofdif}) and
 (\ref{difftoone}), we get
   $$c=\sum_{i=1}^{t-1} \beta_{i} [y_i,y_t,x,x'_4\dots,x'_m].$$
   Therefore,
   $c$ can be represented in the form
 (\ref{centralizer}). This completes the proof of part~3.
   \end{proof}

   Now we are ready to describe centralizers of linear
   combinations of the elements in
   $X$. The following lemma holds:

   \begin{llll}\label{centrintl}
     Let
     $c=c_0+\Delta c$, where
     $c_0$ is a homogeneous component of
     $c$ having the largest multidegree.  If
     $c\in \mathcal{C}\bigl(\sum_{j=1}^l \alpha_j x_{i_j}\bigr)$, where all
     $\alpha_j$ are not equal to
     $0$, then
     $c_0\in \bigcap_{j=1}^l \mathcal{C}\bigl(x_{i_j}\bigr)$.
   \end{llll}
   \begin{proof}
     To be definite assume that
     $$\displaystyle{c\in \mathcal{C}\bigl(\sum_{j=k+1}^n \alpha_j x_j\bigr)}.$$
     So,
     \begin{equation}\label{decommaxmdeg}
       \begin{split}
       0 & = c.\bigl(\sum _{j=k+1}^n \alpha_j x_{j}\bigr)\\
         & = \sum_{j=k+1}^n \alpha_j [c,x_j]\\
               & = \sum_{j=k+1}^{n-1} \alpha_j [c_0,x_j]+ \alpha_{n} [c_0, x_{n}]+
               \sum_{j=k+1}^n \alpha_j [\Delta c,x_j].
       \end{split}
     \end{equation}
     Moreover,
     $[c_0,x_{n}]=0$. Indeed, otherwise, the basis monomials appearing in
     the decomposition of
     $[c_0,x_{n}]$ have the greatest multidegree in the decomposition of
     $c.\bigl(\sum _{j=k+1}^n \alpha_j x_{j}\bigr)$. So, they cannot cancel with other summands. Therefore,
     $c_{0}\in \mathcal{C}(x_{n})$.

     By Lemma~%
\ref{monombycharpc},
     $c_0$ does not contain
     $x_n$ in it, neither does
     $c$. Let
     $X'=X\backslash\{x_n\}$ and let
     $G'$ be the subgraph of
     $G$ generated by
     $X'$. Consider the projection
     $\pi_{X'}$. This projection maps
     $c$ identically to
     $M(X';G')$. Arguing as above, we see that
     $c_0\in \mathcal{C}(x_{n-1})$, and so on. Finally, we get
     $c_0\in \bigcap_{j=k+1}^n \mathcal{C}(x_j)$, that concludes the proof.
   \end{proof}

   \begin{ttt}\label{centinter}
     Let
     $M(X;G)$ be a partially commutative metabelian Lie algebra,
     where
     $X=\{x_1,x_2,\dots,x_n\}$. Then
     $$\mathcal{C}\bigl(\sum_{j=1}^m \alpha_{i_j} x_{i_j}\bigr)=\bigcap_{j=1}^{m} \mathcal{C}(x_{i_j})$$
     for any elements
     $x_{i_1},x_{i_2},\dots, x_{i_m}$ and for any
     $\alpha_{i_1},\alpha_{i_2},\dots,\alpha_{i_m}\in R\backslash\{0\}$.
   \end{ttt}
   \begin{proof}
     It is sufficient to show that
     $\mathcal{C}\bigl(\sum_{j=k+1}^n \alpha_{j} x_{j}\bigr)$, because otherwise, we can just renumber
     the vertices as we did in the proof of Lemma~%
\ref{centrintl}.

     The inclusion
     $$\bigcap_{j=1}^{m} \mathcal{C}(x_{i_j})\subseteq\mathcal{C}\bigl(\sum_{j=1}^m \alpha_{i_j} x_{i_j}\bigr)$$
     is obvious. So, we are left to show that
     $$\mathcal{C}\bigl(\sum_{j=k+1}^n \alpha_{j} x_{j}\bigr)\subseteq \bigcap_{j=k+1}^{n} \mathcal{C}(x_{j}).$$

     Let
     $c\in \mathcal{C}\bigl(\sum_{j=k+1}^n \alpha_{j} x_j\bigr)$.
     Represent
     $c$ as a linear combination of elements in
     $\mathfrak{B}(X;G)$. We obtain
     $c=c_0+\Delta c$, where
     $c_{0}$ is a homogeneous component of
     $c$ having the largest multidegree. By Lemma~%
\ref{centrintl},
     $c_0\in \bigcap_{j=k+1}^n \mathcal{C}(x_j)$. So, we obviously have
     $c_0\in \mathcal{C}\bigl(\sum_{j=k+1}^n \alpha_j x_j\bigr)$. Consequently,
     $\Delta c=c-c_0\in \mathcal{C}\bigl(\sum_{j=k+1}^n \alpha_j x_j\bigr)$ and the decomposition of
     $\Delta c$ to the sum of homogeneous components has one summand less than the similar decomposition of
     $c$. We can apply Lemma~%
\ref{centrintl} for
     $\Delta c$ and so on. After finitely many steps we obtain each homogeneous component of
     $c$ belongs to
     $\bigcap_{j=k+1}^n \mathcal{C}(x_j)$, so does their sum. Since this sum is equal to
     $c$ we are done.
   \end{proof}

   In Sec.~
\ref{univequiv}, we consider partially commutative Lie algebras
   whose defining graphs are trees. Let us prove the following theorem about
   centralizers in this case.

   \begin{ttt}\label{centlincomb}
     Let
     $G$ be a tree and let
     $\alpha_j \in R\backslash\{0\}$ for
     $j=1,2\dots, m$, where
     $m\geqslant 2$. Then
     $\mathcal{C}\left(\sum_{j=1}^{m}\alpha_j x_{i_j}\right)=0$.
   \end{ttt}
   \begin{proof}
     Since
     $\mathcal{C}\left(\sum_{j=1}^{m}\alpha_j x_{i_j}\right)\subseteq
      \mathcal{C}\left(\sum_{j=1}^{m-1}\alpha_j x_{i_j}\right)$,
     it is sufficient to show that
     $\mathcal{C}(\alpha x_s+\beta x_t)=0$, if
     $\alpha,\beta \neq 0$ and
     $s\neq t$.

     If either
     $x_s$ or
     $x_t$ is an endpoint, then Theorem~%
\ref{centrgen} implies that the corresponding
     centralizer is trivial and the statement holds obviously. So, we are left
     to consider the case when neither vertex is an endpoints.

     Let
     $c\in \mathcal{C}(\alpha x_s+\beta x_t)$. By the proof of Theorem
\ref{centinter}, we may assume that
     $c$ is a homogeneous element of
     $M(X;G)$. By Lemma~%
\ref{centrintl}, neither
     $x_s$, nor
     $x_t$ appears in
     $c$. Therefore, these vertices are not in
     $G_c$.

     By
     $\widetilde{X}$ denote the set
     $X_c\cup\{x_s,x_t\}$. Let
     $\widetilde{G}$ be the subgraph of
     $G$ generated by
     $\widetilde{X}$. It is obvious that
     $\widetilde{G}$ is a forest.

     Consider the graph
     $G_c$. This graph is also a forest. Moreover, all vertices adjacent to
     $x_s$ are in different connected components of
     $G_c$. Indeed, let
     $y$ and
     $z$ be adjacent to
     $x_s$ and belong to the same connected component of
     $G_c$. Then there is a simple path
     $(y,w_1,w_2,\dots, w_l,z)$ in this graph. However, since
     $x_s$ is not in
     $G_c$ we obtain the cycle
     $(x_s,y,w_1,w_2,\dots, w_l,z,x_s)$ in
     $G$. This is a contradiction because
     $G$ is a tree. It is obvious that the assertion also holds for the vertices adjacent to
     $x_t$.

     Let us show that there is at most one connected component
     $\Gamma$ of
     $G_c$ satisfying the following property: there are vertices
     $z,z'\in \Gamma$ such that
     $z$ is adjacent to
     $x_s$ and
     $z'$ is adjacent to
     $x_t$. Indeed, if there are at least two such components (denote two of them as
     $\Gamma_1$ and
     $\Gamma_2$), then there exists a cycle going from
     $x_s$ to
     $x_t$ by vertices of
     $\Gamma_1$ and then, backward, from
     $x_t$ to
     $x_s$ by vertices of
     $\Gamma_2$. Since
     $G$ is a tree, we get a contradiction.

     So,
     $G_c$ has at most one connected component having  a vertex adjacent to
     $x_s$ as well as a vertex adjacent to
     $x_t$. If
     $\Gamma$ exists, we can renumber the vertices of
     $G$ in such a way that the minimal vertex (denote it by
     $x$) is in
     $\Gamma$ and each vertex adjacent to either
     $x_s$ or
     $x_t$ and not belonging to
     $\Gamma$ is the largest in its connected component. If there is no component
     $\Gamma$ then any vertex of
     $G_c$ can be chosen as
     $x$.

     Theorem~%
\ref{centinter} implies that
     $c\in \mathcal{C}(x_s)\cap \mathcal{C}(x_t)$. By Theorem~%
\ref{centrgen}, we obtain
     \begin{equation}\label{first}
       c=\sum_{\{x_i,x_s\},\{x_j,x_s\}\in G_{c}}[x_i,x_j].f_{ij}
     \end{equation}
     and
     \begin{equation}\label{second}
       c=\sum_{\{x_p,x_t\},\{x_q,x_t\}\in G_{c}}[x_p,x_q].g_{pq}
     \end{equation}
     where
     $f_{ij},g_{pq}\in R[X_c]$ are such that the multidegree of each monomial in both
     equations is equal to
     $\mdeg(c)$.  If
     $x=x_i$ in the monomial
     $[x_i,x_j].f_{ij}$ of the right-hand side of
(\ref{first}), then it can be written in the form
     \begin{equation}\label{interchange}
       [x_i,x_j].f_{ij}=-[x_j,x].f_{ij}.
     \end{equation}
     If
     $x\neq x_i,x_j$, then put the smallest letter in each monomial to the third position.
     After that, we can transform each such summand in the right-hand side of
(\ref{first}) as follows:
     \begin{equation}\label{basdecompose}
       \begin{split}
         [x_i,x_j].f_{ij}&=\sum_{l} \alpha_l [x_i,x_j,x,y_1,\dots,y_m]\\
         &= \sum_{l}\alpha_l ([x_i,x,x_j,y_1,\dots,y_m]-[x_j,x,x_i,y_1,\dots,y_m])
       \end{split}
     \end{equation}

     Substituting
(\ref{interchange}) and
 (\ref{basdecompose}) to
 (\ref{first}) and combining like terms we obtain a linear combination of basis monomials
     such that the first letter of each of them is adjacent to
     $x_s$ and does not belong to
     $\Gamma$  (if there is the component
     $\Gamma$ in
     $G_c$). Indeed, if, for instance,
     $x_i$ and
     $x$ belong to the same connected component of
     $G_c$ then
     $[x_i,x,x_j,y_1,\dots,y_m]=0$. It means that this summand is not in the basis of
     $M(X;G)$ and we can exclude it from
     the obtained linear combination.
     Let us rewrite each summand in the right-hand side of
(\ref{second}) in a similar way. We obtain a linear combination of
     basis elements, each of which begins with a letter adjacent to
     $x_t$ and not belonging to
     $\Gamma$. Since no vertices outside
     $\Gamma$ can be adjacent to both
     $x_s$ and
     $x_t$ the obtained linear combinations have no equal monomials. Since the set of
     monomials in
(\ref{first}) and
 (\ref{second}) is linearly independent, all coefficients by
     the basis elements in these linear combinations are equal to zero.
     Therefore,
     $c=0$ and the proof is complete.
   \end{proof}

   \section{Universal equivalence}\label{univequiv}
   In this section,
   $M(X;T)$ denotes a partially commutative metabelian Lie ring (i.e. a partially commutative metabelian
   Lie algebra over the ring of integers
   $\mathbb{Z}$) whose defining graph is a tree.

   As above, let
   $X=\{x_1,x_2,\dots x_n\}$. Denote by
   $x$ the smallest vertex in
   $X$ (namely the vertex
   $x_1$).

   We need some auxiliary technical results to prove the main theorem.

   \begin{llll}\label{visvert}
     Let
     $T_1=\langle X,E_1 \rangle$ and
     $T_2=\langle Y,E_2 \rangle$ be trees. If the algebras
     $M(X;T_1)$ and
     $M(Y;T_2)$ are universally equivalent and at least two
     vertices of
     $T_1$ are not endpoints, then at least two vertices of
     $T_2$ are not endpoints either.
   \end{llll}
   \begin{proof}
     Let
     $Y=\{y_1,y_2,\dots,y_m\}$ and set
     $y_i>y_j$ iff
     $i>j$.
     First of all, let us show that if at least two vertices of
     $T_1$ are not endpoints then at least two of them are adjacent. Suppose that
     $x_i$ and
     $x_j$ are not endpoints. If
     $x_i$ and
     $x_j$ are not adjacent in
     $T_1$ then there exists a simple path
     $(x_i,x_{i_1},\dots, x_{i_k},x_j)$ connecting these vertices.
     We are left to notice that all inner vertices of this path
     are not endpoints because the degree of each endpoint is
     $1$ while the degree of any inner vertex is at least
     $2$.

     Therefore, we can choose different vertices
     $z_1,z_2,z_3,z_4$ in
     $X$ in such a way that the
     subgraph
     $\widetilde{T}$ of
     $T_1$ generated by these vertices seems as on Figure~%
\ref{st}.
     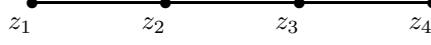
\begin{figure}
      \begin{picture}(170,40)
       \put(10,20){\line(1,0){150}}
       \put(10,20){\circle*{4}}
       \put(60,20){\circle*{4}}
       \put(110,20){\circle*{4}}
       \put(160,20){\circle*{4}}
       \put(1,10){\makebox{$z_1$}}
       \put(51,10){\makebox{$z_2$}}
       \put(101,10){\makebox{$z_3$}}
       \put(151,10){\makebox{$z_4$}}
      \end{picture}
     \caption{Subgraph generated by
     $z_1,z_2,z_3,z_4$} \label{st}
     \end{figure}
     Denote
     $\{z_1,z_2,z_3,z_4\}$ by
     $\widetilde{X}$.

     Since
     $\widetilde{T}$ has no cycles the following formula holds in
     $M(\widetilde{X};\widetilde{T})$ and therefore, in
     $M(X;T_1)$:
     \begin{equation}\label{twonotdung}
       \begin{split}
         \exists v_1v_2v_3v_4 & ([v_1,v_2]=[v_2,v_3]=[v_3,v_4]=0 \\
         & \wedge [v_1,v_3]\neq 0 \wedge [v_1,v_4]\neq 0 \wedge [v_2,v_4] \neq 0).
       \end{split}
     \end{equation}
     We can put
     $v_i$ to be equal to
     $z_i$.

     Suppose that
     $T_2$ has at most one vertex which is  not an endpoint. If
     all vertices of
     $T_2$ are endpoints then
     $Y$ contains one or two
     vertices. In this case
     $M(Y;T_2)$ is an abelian Lie algebra. Consequently, formula
(\ref{twonotdung}) does not hold in this algebra.

     So, without loss of generality we may assume that
     $y_{m}$ is not an endpoint. Then this vertex is adjacent to all other vertices
     and there are no other edges. So,
     \begin{equation}\label{dirsum}
       M(Y;T_2)= M(Y')\oplus \mathcal{L}(y_{m}),
     \end{equation}
     where
     $Y'=Y\backslash\{y_{m}\}$,
     $\mathcal{L}(y_m)$ is one-dimensional Lie algebra, and
     $[u,v]=0$ for any
     $u\in M(Y')$ and  any
     $v \in \mathcal{L}(y_{m})$. It means that
     $M(Y;T_2)$ is a direct sum of a free metabelian Lie algebra and the
     one-dimensional abelian Lie algebra (as vector spaces) and Lie product of any elements from
     different direct summands is equal to zero.

     Suppose that
     $u_1,u_2,u_3 \in M(Y')$. Let us show that if
     $[u_1,u_2]=[u_2,u_3]=0$ then
     $[u_1,u_3]=0$.

     If
     $u_2\not \in M'(Y')$ then Theorem~%
\ref{freeproduct} implies that there exist
     $\alpha,\beta,\gamma,\delta \in \mathbb{Z}\backslash\{0\}$, such that
     $\alpha u_1=\beta u_2$,
     $\gamma u_2=\delta u_3$. Consequently,
     $\alpha\gamma u_1=\beta\gamma u_2=\beta\delta u_3$, where
     $\alpha\gamma,\beta\delta \neq 0$. It means that
     $u_1$ and
     $u_3$ are linearly dependent. Therefore,
     $[u_1,u_3]=0$ by Theorem~%
\ref{freeproduct}.

     Again by Theorem~
 \ref{freeproduct}, it follows that if
     $u_2\in M'(Y')$ then
     $u_1,u_3\in M'(Y')$. Therefore,
     $[u_1,u_3]=0$.

     So, if
     $T_1$ has at least two vertices which are not endpoints, then
(\ref{twonotdung}) holds in
     $M(X;T_1)$. Since
     $M(X;T_1)$ and
     $M(Y;T_2)$ are universally equivalent, formula
(\ref{twonotdung}) holds in
     $M(Y;T_2)$. By
(\ref{dirsum}), each element
     $v_i$ can be represented in the form
     $v_i=u_i+\alpha_i y_{m}$, where
     $u_i\in M(Y')$. Since the multiples of
     $y_{m}$ are in the center of
     $M(Y,T_2)$, we obtain
     $[v_i,v_j]=[u_i,u_j]$.

     Consequently, if
     $[v_1,v_2]=[v_2,v_3]=0$ then
     $[u_1,u_2]=[u_2,u_3]=0$ and we have
     $[u_1,u_3]=0$. Therefore,
     $[v_1,v_3]=0$ and we have a contradiction to~
(\ref{twonotdung}). Thus, at least two vertices of
    $T_2$ are not endpoints vertices.
   \end{proof}
   \begin{llll}\label{charformula}
     Let
     $T$ be a tree and let
     $X$ be the set of its vertices at least
     $k$ of which are not endpoints. Then the following formula holds in
     $M(X;T)$:
     \begin{equation}\label{longformula}
       \begin{split}
         \Phi(T)= & \exists z_1  \dots z_k  u_1\dots u_k v_1\dots v_k
           \left(\bigwedge_{i=1}^k [u_i,v_i]\neq 0
           \wedge \bigwedge _{i=1}^k [u_i,v_i,z_i]=0 \wedge \right.\\
         &\left. \bigwedge_{i\neq j} [u_i,v_i,z_j]\neq 0
           \wedge \bigwedge_{\{x_i,x_j\}\in T} [z_i,z_j]=0
           \wedge \bigwedge_{\{x_i,x_j\}\not\in T} [z_i,z_j]\neq 0\right).
       \end{split}
     \end{equation}
   \end{llll}
   \begin{proof}
     Without loss of generality we may assume that
     $x_1,\dots,x_k$ are not endpoints. These vertices can be chosen for
     $z_1,\dots z_k$ respectively. For each
     $i=1,2,\dots, k$, let us take two different vertices adjacent to
     $x_i$ as
     $u_i$ and
     $v_i$. We have
     $[u_i,v_i]\neq 0$. Indeed, if
     $u_i$ is adjacent to
     $v_i$ then
     $(u_i,v_i,z_i,u_i)$ is a cycle in
     $T$ and we obtain a contradiction. Next,
     $[u_i,v_i,z_i]=[u_i,z_i,v_i]+[u_i,[v_i,z_i]]=0$ since
     $z_i$ is adjacent to both
     $u_i$ and
     $v_i$. On the other hand, if
     $i\neq j$ then
     $[u_i,v_i,z_j]\neq 0$. Indeed, otherwise Lemma~
\ref{zero} implies that
     $u_i$ and
     $v_i$ are in the same connected component of the subgraph of
     $T$ generated by the set
     $\{u_i,v_i,z_j\}$. Consequently, either
     $u_i$ is adjacent to
     $v_i$ or the vertex
     $z_j$ is adjacent to both
     $u_i$ and
     $v_i$. The former case contradicts to the inequality
     $[u_i,v_i]\neq 0$. In the latter case, we have
     $(u_i,z_i,v_i,z_j,u_i)$ is a cycle. Since
     $T$ is a tree we again have a contradiction. Finally,
     it is obvious that
     $[z_i,z_j]=0$ for
     $\{x_i,x_j\}\in T$ and
     $[z_i,z_j]\neq 0$ for
     $\{x_i,x_j\} \not\in T$.
   \end{proof}

   Let
   $G$ be a graph with the set of vertices
   $X$. We say that an endpoint
   $z$ of
   $G$ is \emph{unnecessary} if degree of the vertex adjacent to
   $z$ is at least
   $3$.

   Suppose that
   $T$ has an unnecessary endpoint. Without loss of generality we can assume that
   $x_n$ is such vertex and
   $x$ is adjacent to
   $x_n$ and to
   $x_{n-2},x_{n-1}$ (see~Figure~%
\ref{t}).
   \begin{figure}
     \begin{picture}(320,164)
       \thicklines
       \put(160,160){\line(0,-1){40}}
       \put(160,120){\line(-3,-1){90}}
       \put(160,120){\line(3,-1){90}}
       \put(70,90){\line(-1,-3){20}}
       \put(70,90){\line(-2,-3){40}}
       \put(70,90){\line(2,-3){40}}
       \put(250,90){\line(-1,-3){20}}
       \put(250,90){\line(-2,-3){40}}
       \put(250,90){\line(2,-3){40}}
       \put(160,120){\line(-1,-3){15}}
       \put(160,120){\line(-2,-3){30}}
       \put(160,120){\line(2,-3){30}}
       \put(160,160){\circle*{4}}
       \put(160,120){\circle*{4}}
       \put(70,90){\circle*{4}}
       \put(250,90){\circle*{4}}
       \multiput(65,30)(15,0){3}{\circle*{2}}
       \multiput(245,30)(15,0){3}{\circle*{2}}
       \multiput(158,75)(10,0){3}{\circle*{2}}
       \qbezier[12](42,15)(50,47)(58,15)
       \qbezier[8](42,15)(50,-10)(58,15)
       \qbezier[12](22,15)(30,47)(38,15)
       \qbezier[8](22,15)(30,-10)(38,15)
       \qbezier[12](102,15)(110,47)(118,15)
       \qbezier[8](102,15)(110,-10)(118,15)
       \qbezier[12](222,15)(230,47)(238,15)
       \qbezier[8](222,15)(230,-10)(238,15)
       \qbezier[12](202,15)(210,47)(218,15)
       \qbezier[8](202,15)(210,-10)(218,15)
       \qbezier[12](282,15)(290,47)(298,15)
       \qbezier[8](282,15)(290,-10)(298,15)
       \qbezier[6](123,65)(130,85)(136,65)
       \qbezier[8](123,65)(130,40)(136,65)
       \qbezier[6](138,65)(145,85)(151,65)
       \qbezier[8](138,65)(145,40)(151,65)
       \qbezier[6](183,65)(190,85)(196,65)
       \qbezier[8](183,65)(190,40)(196,65)
       \put(165,160){\makebox{\large $x_n$}}
       \put(75,85){\makebox{\large $x_{n-1}$}}
       \put(165,120){\makebox{\large $x$}}
       \put(255,85){\makebox{\large $x_{n-2}$}}
       \thicklines
     \end{picture}
     \caption{Graph
     $T$} \label{t}
   \end{figure}

   Denote by
   $X'$ the set
   $X\backslash\{x_n\}$. Let,
   $T_{X'}$ be the subtree of
   $T$ generated by
   $X'$.

   For any
   $\lambda,p \in \mathbb{Z}$ we define the map
   $\varphi_{\lambda, p}: X \rightarrow M(X';T_{X'})$ as follows:
   $$\varphi_{\lambda, p}(x_i)=
     \begin{cases}
       x_i, & \text{ if } i\neq n, \\
       \lambda p x_{n-1}+\lambda x_{n-2}, & \text{ if } i=n.
     \end{cases}
   $$
   It is easy to see that this map can be extended up to a homomorphism
   from
   $M(X;T)$ to
   $M(X';T_{X'})$ uniquely. We denote this homomorphism by
   $\varphi_{\lambda, p}$ as well.
   Let
   $\mathbb{Z}^+$ denote the set of all positive integers.

   \begin{llll}\label{poly}
     Let
     $f(x,y)$ be a non-zero polynomial over
     $\mathbb{R}$. Then there exist a number
     $X\in \mathbb{Z}^+$ and a function
     $Y: \mathbb{Z}^+ \rightarrow \mathbb{Z}^+$ such that
     $f(x,y)\neq 0$ for all positive integers
     $x>X$ and
     $y>Y(x)$.
   \end{llll}
   \begin{proof}
     Rewrite
     $f(x,y)$ in the form
     \begin{equation*}
       f(x,y)=\sum_{i=0}^n p_i(x)y^n,
     \end{equation*}
     where
     $p_i(x)$,
     $i=0,1,\dots n$ are polynomials in the indeterminate
     $x$. Let
     $X$ be the smallest positive integer that is greater than all roots of
     all polynomials
     $p_i(x)$. Then for each
     $x_0>X$ we have
     $g(y)=f(x_0,y)$ is a polynomial in
     $y$ with non-zero coefficients. Consequently, one may set
     $Y(x_0)$ to be the smallest positive integer that is greater than the greatest root of
     $g(y)$. One can chose any values for
     $Y(x_0)$ if
     $x_0\leqslant X$.
   \end{proof}

   \begin{llll}\label{tononero}
     Let
     $g \in M(X;T)\backslash M'(X;T)$ and
     $p$ any positive integer; then there exists
     $\lambda_0 \in \mathbb{Z}^+$ such that
     $\varphi_{\lambda, p}(g)\neq 0$ for any
     $\lambda\geqslant \lambda_0$.
   \end{llll}
   \begin{proof}
     Let
     $g=g_0+g_1$, where
     $g_0=\sum_{i=1}^n \alpha_i x_i$  and
     $g_1\in M'(X;G)$. Since
     $g\in M(X;G)\backslash M'(X;G)$ we have
     $g_0\neq 0$.

     If
     $\alpha_n=0$, then
     $\varphi_{\lambda, p}(g)=g_0+\varphi_{\lambda, p}(g_1)$. Since
     $\varphi_{\lambda, p}(g_1)\in M'(X;T)$ one can see that
     $g_0$ and
     $\varphi_{\lambda, p}(g_1)$ are linearly
     independent. Thus,
     $\varphi_{\lambda, p}(g)\neq 0$ for any
     $\lambda$ and we can put, for example,
     $\lambda_0=1$.

     If
     $\alpha_n\neq 0$, then
     $$\varphi_{\lambda, p}(g)=\sum_{i=1}^{n-3} \alpha_i x_i+
       (\alpha_{n-2}+\lambda \alpha_{n})x_{n-2}+
       (\alpha_{n-1}+\lambda p \alpha_{n})x_{n-1}+\varphi_{\lambda, p}(g_1).$$

     Since
     $\mathbb{Z}$ is an integral domain and the polynomial
     $\alpha_{n-2}+\lambda \alpha_{n}$ does not depend on
     $p$ there exists
     $\lambda_0$ satisfying the conditions of the lemma. So,
     $(\alpha_{n-2}+\lambda \alpha_{n})x_{n-2}\neq 0$. Since all generators are linearly
     independent, this summand cannot cancel with
     other summands of the length
     $1$. Therefore,
     $\varphi_{\lambda, p}(g_0)\neq 0$. It means that
     $\varphi_{\lambda, p}(g_0)\not \in M'(X;T)$. Since
     $\varphi_{\lambda, p}(g_1)\in M'(X;T)$, we obtain
     $\varphi_{\lambda, p}(g)=\varphi_{\lambda, p}(g_0)+\varphi_{\lambda, p}(g_1)\neq 0$ for
     any
     $\lambda \geqslant \lambda_0$.
  \end{proof}

  Consider
  $g\in M'(X;T)$. There is the decomposition
  \begin{equation} \label{mdeghomdecomp}
    g=\sum_{\overline{\delta}} g_{\overline{\delta}},
  \end{equation}
  where each
  $g_{\overline{\delta}}$ is a Lie monomial such that
  $\mdeg(g_{\overline{\delta}})=\overline{\delta}$. Denote by
  $\mdeg(v;k)$ the tuple of the first
  $k$ coordinates of
  $\mdeg(v)$. By analogy,
  $\overline{\delta}(k)$ is the tuple
  $(\delta_1,\delta_2,\dots, \delta_k)$.

  \begin{llll}\label{nonhomtohom}
     Let
     $g \in M(X;T)$ and let there be a positive
     integer number
     $N$ such that for any
     $\lambda>N$ the equality
     $\varphi_{\lambda,p}(g)=0$ holds for infinitely many integers
     $p>P(\lambda)$. Then
     $\varphi_{\lambda,p}(g_{\overline{\delta}})=0$ for all
     $\lambda$,
     $p$ and for all
     $g_{\overline{\delta}}$ in decomposition~%
(\ref{mdeghomdecomp}).
   \end{llll}

   \begin{proof}
     Consider
     $g_{\overline{\delta}}$ for some
     $\overline{\delta}$. Let
     $v$ be a basis monomial in the decomposition of
     $g_{\overline{\delta}}$ and
     $v'$ a monomial in the decomposition of
     $\varphi_{\lambda,p}(v)$ as a linear combination of basis
     elements. Then
     $\mdeg(v';n-3)=\overline{\delta}(n-3)$.

     Let
     $v'$ be some monomial of
     $M(X';T')$. Consider a basis monomial
     $v$ of
     $M(X;T)$ such that
     $v'$ is contained in the decomposition of
     $\varphi_{\lambda,p}(v)$ as a linear combination of basis elements. We have
     $\mdeg(v;n-3)=\mdeg(v';n-3)$. Moreover, it is easy to see that
     $\mdeg_{n-2}(v)+\mdeg_{n-1}(v)+\mdeg_{n}(v)=\mdeg_{n-2}(v')+\mdeg_{n-1}(v')$,
     $\mdeg_{n-2}(v)\leqslant \mdeg_{n-2}(v')$, and
     $\mdeg_{n-1}(v)\leqslant \mdeg_{n-1}(v')$. Hence, if
     $\mdeg(v')=(\gamma_1,\dots \gamma_{n-1})$ then
     $\mdeg(v)=(\gamma_1,\dots \gamma_{n-3},\gamma_{n-2}-s,\gamma_{n-1}-r,r+s)$,
     where
     $r,s\geqslant 0$.

     Note that the coefficient of
     $v'$ in the decomposition of
     $\varphi_{\lambda,p}(v)$  is
     $\beta \lambda^{r+s}p^r$, where
     $\beta$ does not depend on
     $\lambda$,
     $p$.

     Let
     $g=\sum \alpha_i u_i$, where
     $u_i$ are different basis monomials. Let us find
     $\varphi_{\lambda,p}(u_i)$ for all
     $u_i$ and compute the sum of coefficients by
     $v'$ in all such monomials. This sum is equal to
     $$\sum_{m=0}^{\gamma_{n-2}+\gamma_{n-1}}\sum_{k=\max(0,m-\gamma_{n-2})}^{\min(m,\gamma_{n-1})} \beta_{m,k}\lambda^{m}p^k,$$
     where the coefficients
     $\beta_{m,k}$ depend on the coefficients
     $\alpha_i$ but do not depend on
     $\lambda$ and
     $p$.

     On the other hand, since
     $\varphi_{\lambda,p}(g)=0$ for
     $\lambda >N$ and for infinitely many positive integers
     $p$, we obtain the coefficient by
     $v'$ is equal to zero. By Lemma~%
\ref{poly}, we get
     $\beta_{m,k}=0$ for all
     $m$ and
     $k$. We are left to note that
     $\beta_{m,k}\lambda^{m}p^k$ is the sum of coefficients by
     $v'$ in the decomposition of the elements
     $\varphi_{\lambda,p}(u_i)$ such that
     $\mdeg (u_i)=(\gamma_1,\dots,\gamma_{n-3},\gamma_{n-2}-m+k,\gamma_{n-1}-k,m)$.
     So, the multidegrees of all such elements are same. Therefore,
     the coefficient by
     $v'$ in the decomposition of
     $\varphi_{\lambda,p}(g_{\overline{\delta}})$ is
     $0$ for any multidegree
     $\overline{\delta}$. Since
     $v'$ is an arbitrary monomial, the proof is complete.
   \end{proof}

   Let
   $w(x_i,\overline{\delta}(n-3),\delta_{n-2},\delta_{n-1},\delta_{n})$
   be the basis monomial of
   $M(X;T)$ such that it starts with
   $x_i$ and its multidegree is
   $\overline{\delta}=(\delta_1,\dots,\delta_n)$. Similarly, let us denote by
   $w'(x_i,\overline{\gamma}(n-3),\gamma_{n-2},\gamma_{n-1})$ the
   basis monomial of
   $M(X';T')$ such that it starts with
   $x_i$ and its multidegree is
   $\overline{\gamma}=(\gamma_1,\dots,\gamma_{n-1})$. By analogy with
   Sec.~%
\ref{pcalgbas}, if
   $g$ is a homogeneous polynomial then we denote by
   $T_g$ the subgraph of
   $T$ generated by
   $\supp(g)$.

   Let
   $g$ be a polynomial in
   $M(X;T)$. Replace all
   $x_n$'s in
   $g$ with
   $\lambda p x_{n-1}$ and represent the obtained polynomial as a linear combination of basis elements.
   We denote this linear combination by
   $\overline{\varphi}_{\lambda,p}(g)$. Analogously, denote by
   $\underline{\varphi}_{\lambda,p}(g)$ the polynomial obtained by replacing all
   $x_n$'s in
   $g$ with
   $\lambda x_{n-2}$ and representing the obtained polynomial as a
   linear combination of basis elements. Obviously,
   $\varphi_{\lambda,p}(g)=\overline{\varphi}_{\lambda,p}(g)+\underline{\varphi}_{\lambda,p}(g)+h$.
   Here, if
   $\overline{\varphi}_{\lambda,p}(g)\neq 0$ and
   $h\neq 0$ then
   $\mdeg(\overline{\varphi}_{\lambda,p}(g))>\mdeg(w)$ for any monomial
   $w$ appearing in
   $h$.
   Similarly, if
   $\underline{\varphi}_{\lambda,p}(g)\neq 0$ and
   $h\neq 0$ then
   $\mdeg(\underline{\varphi}_{\lambda,p}(g))<\mdeg(w)$
   for any monomial
   $w$ appearing in
   $h$. Consequently, if
   $\varphi_{\lambda,p}(g)=0$, then
   $\overline{\varphi}_{\lambda,p}(g)=\underline{\varphi}_{\lambda,p}(g)=0$.

   \begin{llll}\label{tononzero}
    For any
    $g \in M'(X;T)\backslash\{0\}$ there exist positive integers
    $\lambda_0$ and
    $p_0$ such that
    $p_0=p_0(g)$,
    $\lambda_0=\lambda_0(g,p_0)$, and
    $\varphi_{\lambda, p}(g)\neq 0$ for any
    $\lambda\geqslant \lambda_0$ and
    $p\geqslant p_0$.
  \end{llll}

  \begin{proof}
     By Lemma~%
\ref{nonhomtohom}, we are left to consider the case when
     $g$ is homogeneous.

     Obviously, if
     $\varphi_{\lambda,p}(g.y)\neq 0$ for a generator
     $y$, then
     $\varphi_{\lambda,p}(g)\neq 0$.

     Let
     $g\neq 0$. Without loss of generality we may suppose that
     $\supp(g)$ contains all generators but one of them. Indeed, since
     $T$ is connected we have
     $\supp (g)\neq X$ by Theorem
\ref{baspcom}. On the other hand, if
     $|\supp(g)|<n-1$ then Theorem
\ref{centrintl} and Theorem
 \ref{centlincomb} imply there exists
     $y\in X\backslash \supp(g)$ such that
     $g.y\neq 0$. If
     $|\supp(g.y)|<n-1$, then there exists
     $z$, such that
     $(g.y).z\neq 0$ and so on. Finally, we obtain an element
     $g'\neq 0$ having all generators but one of them and
     $\varphi_{\lambda,p}(g')\neq 0$ implies
     $\varphi_{\lambda,p}(g)\neq 0$.

     So, let
     $|\supp(u)|=n-1$ and
     $\varphi_{\lambda,p}(g)=0$. There are several cases possible.

     \noindent
     1. If
     $\mdeg_n(g)=0$, namely if
     $x_n\not \in \supp(g)$, then we have
     $\varphi_{\lambda,p}(g)=g \neq 0$ for any
     $\lambda$ и
     $p$.\\

     \noindent
     2. Let
     $x\not \in \supp(g)$, then
     $x_{n-1},x_{n-2}\in \supp(g)$. It is obvious that
     $T_g$ is a forest. Denote by
     $T_0,T_1,\dots, T_r$ its connected components, where
     $T_0=\{x_n\}$, and
     $y_{1}=x_{n-1},y_{2}=x_{n-2}, y_{3}, \dots, y_{r}$ are the greatest vertices of
     $T_1,T_2,\dots, T_r$ respectively. Note that if
     $v$ is a monomial appearing in
     $g$ then the subgraph of
     $T$ generated by
     $\supp(\varphi_{\lambda,p}(g))$ is of the form
     $\displaystyle{\bigcup_{i=1}^r}\,T_i$. In particular its connected components are
     $T_1,T_2,\dots, T_r$. \\

     \noindent
     2.1. Let the smallest vertex of
     $T_g$ be in
     $T_1$. Since
     $n\geqslant 4$,
     $T_g$ has at least three vertices. Therefore, the smallest vertex of
     $T_g$ is not equal to
     $x_{n-1}$. We have
     \begin{equation*}
       \begin{split}
          g & = \alpha_0 w(x_n,\delta(n-3),\delta_{n-2},\delta_{n-1},\delta_n)
              + \alpha_2 w(x_{n-2},\delta(n-3),\delta_{n-2},\delta_{n-1},\delta_n)\\
            & + \sum_{i=3}^r\alpha_i w(y_i,\delta(n-3),\delta_{n-2},\delta_{n-1},\delta_n).
       \end{split}
     \end{equation*}
     Hence
     \begin{eqnarray}
       \begin{split}\label{21d}
          \underline{\varphi}_{\lambda,p}(g)
            & = \lambda^{\delta_n}(\alpha_0+\alpha_2)w'(x_{n-2},\delta(n-3),\delta_{n-2}+\delta_n,\delta_{n-1})\\
            &  +  \lambda^{\delta_n} \sum_{i=3}^r\alpha_i w'(y_i,\delta(n-3),\delta_{n-2}+\delta_n,\delta_{n-1});
       \end{split}\\
       \begin{split}
          \overline{\varphi}_{\lambda,p}(g)\label{21u}
            & = \lambda^{\delta_n}p^{\delta_n} \alpha_2 w'(x_{n-2},\delta(n-3),\delta_{n-2},\delta_{n-1}+\delta_n)\\
            &  +  \lambda^{\delta_n} p^{\delta_n}\sum_{i=3}^r\alpha_i w'(y_i,\delta(n-3),\delta_{n-2},\delta_{n-1}+\delta_n).
       \end{split}
     \end{eqnarray}
     Thus, the right-hand sides of
(\ref{21d}) and
 (\ref{21u}) are linear combinations of basis elements
     Therefore,
     $\alpha_i=0$ for
     $i=2,3,\dots,r$ and
     $\alpha_0+\alpha_2=0$. Consequently,
     $\alpha_0=0$ as well. So,
     $g=0$, we get a contradiction.\\

     \noindent
     2.2. Let the smallest vertex of
     $T_g$ be in
     $T_2$. We have
     \begin{equation*}
       \begin{split}
          g & = \alpha_0 w(x_n,\delta(n-3),\delta_{n-2},\delta_{n-1},\delta_n)
              + \alpha_1 w(x_{n-1},\delta(n-3),\delta_{n-2},\delta_{n-1},\delta_n)\\
            & + \sum_{i=3}^r\alpha_i w(y_i,\delta(n-3),\delta_{n-2},\delta_{n-1},\delta_n).
       \end{split}
     \end{equation*}
     Hence,
     \begin{eqnarray}
       \begin{split}\label{22d}
          \underline{\varphi}_{\lambda,p}(g)
            & = \lambda^{\delta_n}\alpha_1 w'(x_{n-1},\delta(n-3),\delta_{n-2}+\delta_n,\delta_{n-1})\\
            &  +  \lambda^{\delta_n} \sum_{i=3}^r\alpha_i w'(y_i,\delta(n-3),\delta_{n-2}+\delta_n,\delta_{n-1});
       \end{split}\\
       \begin{split}
          \overline{\varphi}_{\lambda,p}(g)\label{22u}
            & = \lambda^{\delta_n}p^{\delta_n} (\alpha_0+\alpha_1) w'(x_{n-2},\delta(n-3),\delta_{n-2},\delta_{n-1}+\delta_n)\\
            &  +  \lambda^{\delta_n} p^{\delta_n}\sum_{i=3}^r\alpha_i w'(y_i,\delta(n-3),\delta_{n-2},\delta_{n-1}+\delta_n).
       \end{split}
     \end{eqnarray}
     The right-hand sides of
(\ref{22d}) and
 (\ref{22u}) are linear combinations of basis elements.
     Consequently,
     $\alpha_i=0$ for
     $i=1,3,4,\dots,r$ and
     $\alpha_0+\alpha_1=0$. Therefore,
     $\alpha_0=0$. So,
     $g=0$. It is again a contradiction.\\

     \noindent
     2.3. Suppose that the smallest vertex of
     $T_g$ is neither in
     $T_1$ nor in
     $T_2$. Without loss of generality we may suppose that it is
     in
     $T_r$. We have
     \begin{equation*}
       \begin{split}
          g & = \alpha_0 w(x_n,\delta(n-3),\delta_{n-2},\delta_{n-1},\delta_n)
              + \alpha_1 w(x_{n-1},\delta(n-3),\delta_{n-2},\delta_{n-1},\delta_n)\\
            & + \alpha_2 w(x_{n-2},\delta(n-3),\delta_{n-2},\delta_{n-1},\delta_n)
              + \sum_{i=3}^{r-1}\alpha_i w(y_i,\delta(n-3),\delta_{n-2},\delta_{n-1},\delta_n).
       \end{split}
     \end{equation*}
     Therefore,
     \begin{eqnarray}
       \begin{split}\label{23d}
          \underline{\varphi}_{\lambda,p}(g)
            & = \lambda^{\delta_n}\alpha_1 w'(x_{n-1},\delta(n-3),\delta_{n-2}+\delta_n,\delta_{n-1})\\
            &  + \lambda^{\delta_n}(\alpha_0+\alpha_2) w'(x_{n-2},\delta(n-3),\delta_{n-2}+\delta_n,\delta_{n-1})\\
            &  +  \lambda^{\delta_n} \sum_{i=3}^{r-1} \alpha_i w'(y_i,\delta(n-3),\delta_{n-2}+\delta_n,\delta_{n-1});
       \end{split}\\
       \begin{split}\label{23u}
          \overline{\varphi}_{\lambda,p}(g)
            & =\lambda^{\delta_n}p^{\delta_n}(\alpha_0+\alpha_1) w'(x_{n-1},\delta(n-3),\delta_{n-2},\delta_{n-1}+\delta_n)\\
            &  + \lambda^{\delta_n}p^{\delta_n}\alpha_2 w'(x_{n-2},\delta(n-3),\delta_{n-2},\delta_{n-1}+\delta_n)\\
            &  +  \lambda^{\delta_n} p^{\delta_n}\sum_{i=3}^{r-1} \alpha_i w'(y_i,\delta(n-3),\delta_{n-2},\delta_{n-1}+\delta_{n}).
       \end{split}
     \end{eqnarray}
     The right-hand sides of
(\ref{23d}) and
 (\ref{23u}) are linear combinations of basis elements.
     Hence,
     $\alpha_i=0$ for
     $i=1,2,\dots,r-1$ and
     $\alpha_0+\alpha_1=0$. Consequently,
     $\alpha_0=0$ as well. As in two last cases,
     $g=0$, a contradiction.\\

     \noindent
     3. Suppose that
     $x,x_n\in \supp(g)$. Let
     $T_0,T_1, \dots, T_r$ be connected components of
     $T_g$ and
     $y_0=x_n,y_1,\dots y_r$ the greatest vertices of the corresponding components. In this case,
     $x$ is the smallest vertex of
     $\supp(g)$ and
     $x\in T_0$. Let
     $\widetilde{T}$ be the subgraph of
     $T$ generated by
     $\supp(\varphi_{\lambda,p}(g))$. Then the connected components of
     $\widetilde{T}$ are
     $T'_0, T_1,T_2,\dots,T_r$, where
     $T'_0$ is a subgraph of
     $T_0$ by deleting
     $x_n$ (and the incident edge
     $\{x_n,x\}$). We obtain the following representation:
     \begin{equation}
       g=\sum_{i=1}^r \alpha_i w(y_i,\delta(n-3),\delta_{n-2},\delta_{n-1},\delta_n).
     \end{equation}
     Consequently,
     \begin{equation} \label{3u}
       \overline{\varphi}_{\lambda,p}(g)=
         \lambda^{\delta_n}p^{\delta_n}\sum_{i=1}^r \alpha_i w'(y_i,\delta(n-3),\delta_{n-2},\delta_{n-1}+\delta_n).
     \end{equation}
     The right-hand side of
(\ref{3u}) is a linear combination of basis elements.
     Therefore,
     $\alpha_i=0$ for
     $i=1,2,\dots,r$. Thus,
     $g=0$ and we get a contradiction.
   \end{proof}

  Let
  $X^*$ denote the set of all vertices of
  $T$ that are not endpoints and let
  $T^*$ be a subgraph of
  $T$ generated by
  $X^*$.

  Now, we describe the procedure constructing a tree
  $T'$ with no unnecessary vertices for any tree
  $T$. If there are unnecessary vertices in
  $T$ then we delete one of them as well as the edge incident to it. If the obtained graph still
  has unnecessary vertices we again delete one of them. Let us continue this procedure
  until the resulting graph has no unnecessary
  vertices. Denote by
  $T'$ the obtained tree and by
  $X'$ the set of its vertices (Figure~%
\ref{t't*} gives an example).

  \begin{figure}
  \begin{picture}(300,200)
    \put(40,20){\line(5,3){30}}
    \put(40,20){\circle*{4}}
    \put(68,38){\circle*{4}}
    \put(68,38){\line (1,1){30}}
    \put(68,38){\line (2,1){60}}
    \put(98,68){\circle*{4}}
    \put(128,68){\circle*{4}}
    \put(40,20){\line(-2,3){12}}
    \put(28,38){\circle*{4}}
    \put(28,38){\line(-2,3){20}}
    \put(28,38){\line(0,3){30}}
    \put(28,38){\line(2,3){20}}
    \put(8,68){\circle*{4}}
    \put(28,68){\circle*{4}}
    \put(48,68){\circle*{4}}
    \put(8,68){\line(0,3){30}}
    \put(8,97){\circle*{4}}
    \put(40,1){\makebox(0,0){$T$}}
    \put(240,20){\line(5,3){30}}
    \put(240,20){\circle*{4}}
    \put(268,38){\circle*{4}}
    \put(268,38){\line (1,1){30}}
    \put(298,68){\circle*{4}}
    \put(240,20){\line(-2,3){12}}
    \put(228,38){\circle*{4}}
    \put(228,38){\line(-2,3){20}}
    \put(208,68){\circle*{4}}
    \put(208,68){\line(0,3){30}}
    \put(208,97){\circle*{4}}
    \put(240,1){\makebox(0,0){$T'$}}
    \put(140,120){\line(5,3){30}}
    \put(140,120){\circle*{4}}
    \put(168,138){\circle*{4}}
    \put(140,120){\line(-2,3){12}}
    \put(128,138){\circle*{4}}
    \put(128,138){\line(-2,3){20}}
    \put(108,168){\circle*{4}}
    \put(140,100){\makebox(0,0){$T^*$}}
  \end{picture}
  \caption{Graphs
  $T'$ and
  $T^*$.}
  \label{t't*}
  \end{figure}
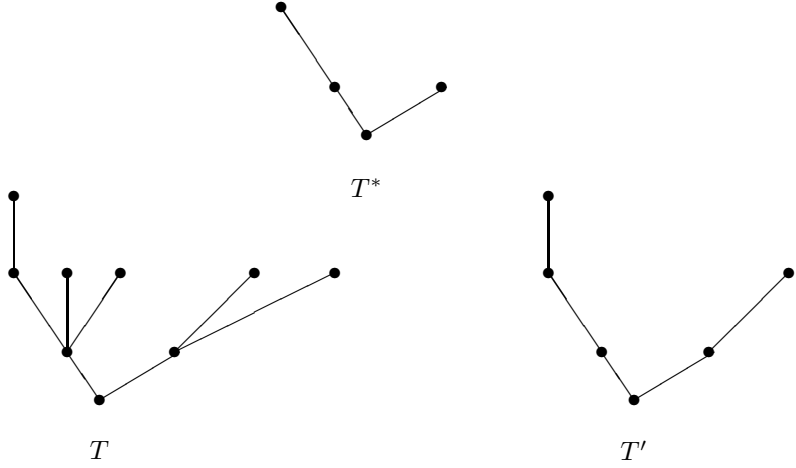

  \noindent
  It is easy to see that for any graph
  $T$ the obtained graphs with no unnecessary vertices are isomorphic no matter what vertices are deleted.

  \begin{llll}\label{homfamily}
    Let
    $L_1$ and
    $L_2$ be two Lie algebras and let
    $\varphi_{\lambda, p}: L_1 \rightarrow L_2$ be the set of homomorphisms with
    integer positive parameters
    $\lambda$ and
    $p$. Suppose that for any
    $g\in L_1$ there are integers
    $p_0$ and
    $\lambda_0=\lambda_0 (p_0)$ such that
    $\varphi_{\lambda, p}(g)\neq 0$
    whenever
    $p\geqslant p_0$ and
    $\lambda\geqslant \lambda_0$. Then for any finite set
    $\Gamma=\{g_1, \dots, g_m\}$ of elements in
    $L_1$ there exist integers
    $\lambda$ and
    $p$ such that the restriction of
    $\varphi_{\lambda, p}$ on
    $\Gamma$ (with partial Lie algebra operations) is an inclusion.
   \end{llll}
   \begin{proof}
     Let us extend
     $G=\{g_1,\dots,g_m\}$ adding the elements
     $g_i-g_j$,
     $g_i+g_j-g_k$,
     $[g_i,g_j]-g_k$ for all
     $i,j,k=1,2\dots ,m$. Denote by
     $\overline{\Gamma}$ the obtained set. It is sufficient to show that there
     exist
     $\lambda$ and
     $p$ such that the kernel of
     $\varphi_{\lambda,p}: L_1\rightarrow L_2$ is disjoint with
     $\overline{\Gamma}$. Indeed, if it is the case then the images of the elements in
     $\Gamma$ are distinct. Moreover, if
     $g_i\neq g_j+g_k$ or
     $g_i\neq [g_j,g_k]$ then the images of
     $g_i$ and
     $g_j+g_k$ (images of
     $g_i$ and
     $[g_j,g_k]$ respectively) are not equal either.

     By Lemmas~
\ref{tononero} and
 \ref{tononzero}, for each
     $g\in \overline{\Gamma}$ one can choose
     $p_0(g)$ and
     $\lambda_0 (g,p)$ such that
     $\varphi_{\lambda,p}(g)\neq 0$ for any
     $p\geqslant p_0(g)$ and
     $\lambda\geqslant \lambda_0(g,p)$. First, let us choose
     $p_0(g)$ for each
     $g\in \overline{\Gamma}$. Denote by
     $p_0$ the greatest value out of
     $p_0(g)$. For this
     $p_0$, we find
     $\lambda_{0}(g,p_0)$ for each
     $g$. Let
     $\lambda_0$ be the greatest number among them. Then, for any
     $\lambda\geqslant\lambda_0$,
     $p\geqslant p_0$, and
     $g\in \overline{\Gamma}$ we obtain
     $\varphi_{\lambda,p}(g)\neq 0$. The proof is complete.
   \end{proof}

   \begin{llll}
     $T_1^* \simeq T_2^*$ if and only if
     $T'_1\simeq T'_2$.
   \end{llll}
   \begin{proof}
     Note that if
     $T$ is a tree then so is
     $T^*$. The graph
     $T'$ is obtained from the graph
     $T^*$ as follows: for each endpoint
     $x$ of the graph
     $T^*$, we add a new vertex
     $y$  and the edge
     $\{x,y\}$. Indeed, all vertices of
     $T^*$ are not endpoints in
     $T$. It means that their degrees in
     $T'$ are at least
     $2$. If an edge of
     $T'$ is not contained in
     $T^*$ then it cannot connect two vertices of
     $X^*$, otherwise, there is a cycle in
     $T'$. Moreover, different endpoints of
     $T'$ are adjacent to different vertices, because each vertex of
     $X^*$ cannot be adjacent to more than one endpoint in
     $T'$. We are left to notice that if
     $x$ is not an endpoint in
     $T^*$ then its degree in
     $T^*$ is at least
     $2$. It means that if
     $T$ has an endpoint adjacent to
     $x$, then there are no these vertices in
     $T'$.

     Let
     $T^*_1 \simeq T^*_2$. Since
     $T'$ can be obtained from
     $T^*$ uniquely, we have
     $T'_1 \simeq T'_2$.

     The inverse statement is obvious, i.e. if
     $T'_1 \simeq T'_2$ then
     $T^*_1\simeq T^*_2$.
   \end{proof}

   \begin{ttt}\label{main}
     Let
     $T_1$ and
     $T_2$ be trees with the sets of vertices
     $X=\{x_1,x_2,\dots x_{n_1}\}$ and
     $Y=\{y_1,y_2,\dots, y_{n_2}\}$ respectively,
     $n_1,n_2\geqslant 2$. The algebras
     $M(X;T_1)$ and
     $M(Y;T_2)$ are universally equivalent if and only if
     the trees
     $T^*_1$ and
     $T^*_2$ are isomorphic.
   \end{ttt}
   \begin{proof}
     Let the algebras
     $M(X;T_1)$ and
     $M(Y;T_2)$ be universally equivalent. We prove that
     $T^*_1$ and
     $T^*_2$ are isomorphic.

     If all vertices of
     $T_1$ are endpoints then
     $n_1=2$. Consequently,
     $M(X;T_1)$ is the free two-generated abelian Lie algebra. So, the formula
     \begin{equation}\label{commut}
        \exists v_1 v_2([v_1,v_2]\neq 0)
     \end{equation}
     does not hold in
     $M(X;T_1)$. Therefore, this formula does not hold in
     $M(Y;T_2)$ either. But if
     $n_2\geqslant 3$ then
     $T_2$ is not a complete graph (because
     $T_2$ is a tree). Consequently,
     $[y_{i_1},y_{i_2}]\neq 0$ for some vertices
     $y_{i_1},y_{i_2} \in Y$. So, formula
(\ref{commut}) holds. We get a contradiction. Therefore,
     $n_2=2$ but all trees with two vertices are isomorphic. It
     means that
     $T_1 \simeq T_2$, therefore,
     $T^*_1 \simeq T^*_2$ and we are done. Note that, in this case,
     $T^*_1$ and
     $T^*_2$ are empty graphs.

     Let each graph
     $T_1$ and
     $T_2$ have exactly one vertex that is not an endpoint. Then
     $|X^*|=|Y^*|=1$ and the graphs
     $T^*_1$ and
     $T^*_2$ are isomorphic.

     Lemma
\ref{visvert} implies that if
     $m\geqslant 2$ vertices of
     $T_1$ are not endpoints then at least two vertices of
     $T_2$ are not endpoints.

     By Lemma~%
\ref{charformula}, the formula
     $\Phi(T_1)$ constructed by
     $T_1$ (see~%
(\ref{longformula})) holds in
     $M(X;T_1)$. Since
     $M(X;T_1)$ and
     $M(Y;T_2)$ are universally equivalent,
     $\Phi(T_1)$ also holds in
     $M(Y;T_2)$. Elements
     $z_i$ can be written as
     $z_i=d_i+c_i$, where
     $d_i=\sum_{j=1}^n\alpha_{i,j}y_j$ for some
     $\alpha_{i,j}\in \mathbb{Z}$ and
     $c_i\in M'(Y;T_2)$. Since
     $[u_i,v_i]\in M'(Y;T_2)$, we obtain
     $[u_i,v_i,z_i]=[u_i,v_i,d_i+c_i]=[u_i,v_i,d_i]$.
     Therefore,
     $[u_i,v_i]\in \mathcal{C}(d_i)$. But if the sum
     $d_i=\sum_{j=1}^n\alpha_{ij}y_j$ has at least two non-zero summands, then
     $\mathcal{C}(d_i)=0$. On the other hand, if
     $d_i=0$ (i.e.
     $z_i\in M'(Y;T_2)$), then
     $[u_s,v_s,z_i]=0$ for each
     $s\neq i$. Therefore,
     $z_i=\beta_i y_{k_i}+c_i$, where
     $\beta_i\neq 0$.

     Let us prove some conditions of the elements
     $z_i$ in
     $M(Y;T_2)$.\\[-0.8em]

     \noindent
     1. If
     $i\neq j$,  then
     $y_{k_i}\neq y_{k_j}$.

     Indeed, if the equality
     $y_{k_i}=y_{k_j}$ holds for some
     $i$ and
     $j$, then we obtain
     $0=[u_i,v_i,z_i]=[u_i,v_i,\beta_i y_{k_i}+c_i]=
     \beta_i[u_i,v_i,y_{k_i}]$. So,
     $[u_i,v_i,y_{k_i}]=0$. On the other hand, using the same arguments for
     $[u_i,v_i,z_j]$ gives
     $0\neq[u_i,v_i,y_{k_j}]=[u_i,v_i,y_{k_i}]$. We get a contradiction.\\[-0.8em]

     \noindent
     2. The vertices
     $y_{k_1},y_{k_2}\dots, y_{k_m}$ are not endpoints.

     Let
     $y_{k_i}$ be an endpoint. We have
     $0=[u_i,v_i,z_i]=[u_i,v_i,\beta_i y_{k_i}+c_i]=\beta_i [u_i,v_i,y_{k_i}]$. Consequently,
     $[u_i,v_i]\in \mathcal{C}(y_{k_i})$. So, by Theorem,~%
\ref{centrgen}
     $[u_i,v_i]=0$. It is a contradiction.\\[-0.8em]

     \noindent
     3. If
     $x_i$ and
     $x_j$ are adjacent in
     $T_1$, then
     $y_{k_i}$ and
     $y_{k_j}$ are adjacent in
     $T_2$.

     If there is the edge
     $\{x_i,x_j\}$ in
     $T_1$, then
     $[z_i,z_j]=0$. Therefore, we have in
     $M(Y;T_2)$:
     \begin{equation}
       \begin{split}
        0 &=[\beta_i y_{k_i}+c_i,\beta_j y_{k_j}+c_j]\\
          &=\beta_i\beta_j[y_{k_i},y_{k_j}]+\beta_i[y_{k_i},c_j]+\beta_j[c_i,y_{k_j}]
       \end{split}
     \end{equation}
     Since the lengthes of all summands in representations of
     $[y_{k_i},c_j]$ and
     $\beta_j[c_i,y_{k_j}]$ as linear combinations of basis elements are greater than 2,
     $[y_{k_i},y_{k_j}]=0$. This means that
     $y_{k_i}$ and
     $y_{k_j}$ are adjacent in
     $T_2$.

     Consider the map
     $\varphi: T^*_1 \rightarrow T^*_2$ defined as follows:
     $\varphi(x_i)=y_{k_i}$. By the properties proved above
     $\varphi$ is an injective homomorphism. Interchanging
     $T_1$ and
     $T_2$ and using the same arguments, we obtain there exists an
     injective homomorphism
     $\psi: T^*_2 \rightarrow T^*_1$. So, in particular
     $|X^*|=|Y^*|$. Therefore,
     $\varphi$ is an isomorphism of the trees
     $T^*_1$ and
     $T^*_2$.

     Conversely, suppose that
     $T_1^*\simeq T_2^*$. Let us show that
     $M(X;T_1)$ and
     $M(Y;T_2)$ are universally equivalent.

     Suppose that the tree
     $\widetilde{T}$ is obtained from the tree
     $T$ by deleting the unnecessary vertex
     $x_n$ and that
     $\widetilde{X}=X\backslash \{x_n\}$. Let us show that the universal theories of
     $M(X;T)$ and
     $M(\widetilde{X};\widetilde{T})$ coinside. By Theorem~%
\ref{univeq}, we are left
     to prove that any finite submodel of
     $M(X;T)$ has an isomorphic submodel in
     $M(\widetilde{X};\widetilde{T})$ and vice versa.

     In one direction the statement is obviously true because
     $M(\widetilde{X};\widetilde{T})$ is a subalgebra of
     $M(X;T)$. Therefore, any finite submodel
     $\{g_1,\dots,g_m\}$ of
     $M(\widetilde{X};\widetilde{T})$ can be considered as a finite submodel of
     $M(X;T)$. In the other direction, the statement holds by Lemma~%
\ref{homfamily}.

     So, the universal theories of
     $M(X;T)$ and
     $M(\widetilde{X};\widetilde{T})$ coinside.

     Consider the trees
     $T_1$ and
     $T_2$. Deleting unnecessary vertices of these graphs one by one, we obtain the graphs
     $T'_1$ and
     $T'_2$. As we just proved, the algebras
     $M(X';T'_1)$ and
     $M(X;T_1)$ are universally equivalent. Analogously,
     the algebras
     $M(Y';T'_2)$ and
     $M(Y;T_2)$ are also universally equivalent. Since
     $T^*_1\simeq T^*_2$ we have
     $T'_1\simeq T'_2$ as was shown above. Therefore, the universal theories of
     $M(X';T'_1)$ and
     $M(Y';T'_2)$ coincide. Consequently, the algebras
     $M(X;T_1)$ and
     $M(Y;T_2)$ are universally equivalent.
   \end{proof}

   Let us note that partially commutative metabelian groups and their theories were studied in
\cite{GT09,Tim10,GT11}. In
 \cite{Tim10}, the criterium of
   universal equivalence for partially
   commutative metabelian groups whose defining graphs are trees was obtained. The obtained there
   necessary and sufficient conditions for universal equivalence of groups are same as the conditions
   of Theorem~%
\ref{main}.

 \end{document}